\documentclass[11pt,a4paper]{amsart}
\usepackage{amssymb,amsmath,amsthm}
\usepackage{extarrows}
\usepackage{color}

\newcommand{\R}{\mathbb{R}}
\newcommand{\N}{\mathbb{N}}

\newcommand{\J}{\mathcal{J}}
\newcommand{\F}{\mathcal{F}}
\newcommand{\T}{\mathcal{T}}

\newtheorem{theorem}{Theorem}[section]
\newtheorem{lemma}[theorem]{Lemma}
\newtheorem{proposition}[theorem]{Proposition}
\newtheorem{corollary}[theorem]{Corollary}
\newtheorem{remark}{Remark}[section]
\theoremstyle{definition}

\numberwithin{equation}{section}
\def\qed{\hfill $\Box$ \smallskip}
\title{ Improved Beckner's inequality for axially symmetric functions on $\mathbb{S}^4$  }

\author{Changfeng Gui}
\address{Changfeng Gui, The University of Texas at San Antonio, TX, USA.
}
\email{changfeng.gui@utsa.edu}
\author{Yeyao Hu}
\address{Yeyao Hu, School of Mathematics and Statistics, HNP-LAMA, Central South University,  Changsha, Hunan 410083, PR China}
\email{huyeyao@gmail.com }
\author{Weihong Xie}
\address{Weihong Xie, School of Mathematics and Statistics, HNP-LAMA, Central South University,  Changsha, Hunan 410083, PR China}
\email{wh.xie@csu.edu.cn}

\begin{document}

\begin{abstract}
   We show that  axially symmetric solutions on $\mathbb{S}^4$ to a constant  $Q$-curvature type equation (it may also be called fourth order mean field equation)  must be constant, provided  that  the parameter $\alpha$  in front of the Paneitz operator belongs to $[\frac{473 + \sqrt{209329}}{1800}\approx0.517, 1)$. This is in contrast to the case $\alpha=1$,  where a family of solutions exist,  known as standard bubbles.  The phenomenon resembles the Gaussian curvature equation on $ \mathbb{S}^2$. As a consequence, we prove an improved Beckner's inequality on $\mathbb{S}^4$ for axially symmetric functions with their centers of mass at the origin. Furthermore, we show uniqueness of axially symmetric solutions when $\alpha=\frac15$ by exploiting Pohozaev-type identities, and prove existence of a non-constant axially symmetric solution for  $\alpha \in (\frac15, \frac12)$ via a bifurcation method.
\end{abstract}

\date{}

\maketitle

{\bf Key words}: Axial Symmetry,  Beckner's inequality, Q-curvature  equation, Paneitz operator, Conformal metrics, Szeg\"o limit theorem,  Bifurcation
\vskip6mm

{\bf 2010 AMS subject classification}: 53A05, 53C18,  53C21,  35J15, 35J35, 35J60,

\vskip6mm

\vskip6mm

{\section{Introduction}}
    We shall consider the constant  $Q$-curvature-type  equation
    \begin{equation}\label{11}
    \alpha P_4 u + 6\left(1-\frac{e^{4u}}{\int_{\mathbb{S}^4}e^{4u}dw}\right)=0, \quad \xi\in \mathbb{S}^4.
    \end{equation}
Here $\mathbb{S}^4$ is the 4-dimensional sphere,
 \begin{equation*}
   P_4=\Delta^2-2\Delta
 \end{equation*}
 is the Paneitz operator on $\mathbb{S}^4$ and $\alpha$ is a positive constant.  The volume form  $dw$ is  normalized so that
$\int_{\mathbb{S}^4} dw=1$.

 The corresponding energy functional is defined in $H^{2}(\mathbb{S}^4)$ as
\begin{equation*}
J_\alpha(u)=\frac{\alpha}{2}\left(\int_{\mathbb{S}^4}|\Delta u|^2dw+2\int_{\mathbb{S}^4}|\nabla u|^2dw\right)+6\int_{\mathbb{S}^4} udw-
\frac{3}{2}\ln \int_{\mathbb{S}^4}e^{4u}dw.
\end{equation*}

When $\alpha =1$, \eqref{11}  corresponds to the constant $Q$-curvature equation on $\mathbb{S}^4$ and it is well known that there holds  the following Beckner's inequality,  a higher order Moser-Trudinger type inequality,
\begin{equation}
\label{beckner}
J_{1} \ge 0, \quad u \in H^2(\mathbb{S}^4).
\end{equation}
Furthermore, $J_1$ is invariant under the
conformal transformation
$$ u(\xi) \to v(\tau \xi)+  \frac14 ln( |det(d\tau)(\xi)|), $$
where $\tau$ is an element of the conformal group of $\mathbb{S}^4$ and  $ |det(\cdot)|$ is the modulus of the corresponding Jacobian determinant. Equality in \eqref{beckner} is only attained at functions of the form
$$
u(\xi)= -\ln(1- \zeta \cdot \xi)+ C, \quad   C  \in \mathbb{R},
$$
where $\zeta \in B^5:=\{ \xi\in \mathbb{R}^5, |\xi| <1\}$.
(See \cite{Beckner}, \cite{Chang95}.)
This in particular implies that \eqref{11} has a family of axially symmetric solutions
$u(\xi)=-\ln(1-a \xi_1), \quad \xi\in \mathbb{S}^4$ for $a\in (-1, 1)$.

On the other hand, an improved Aubin-type inequality holds as shown in \cite[Lemma 4.6]{Chang95} or \cite[Lemma 4.3]{BCY} : for any $\alpha>1/2$, there exists a constant $C(\alpha)\ge 0$ such that  $J_{\alpha}(u) \ge -C(\alpha)$ if  $u$ belongs to  the  set  of  functions  with center of mass at the origin
 \begin{equation*}
   \mathfrak L=\{v\in H^2(\mathbb{S}^4):\int_{\mathbb{S}^4}e^{4v}\xi_jdw=0;\ j=1,2,\cdots,5\}.
 \end{equation*}

This leads to the existence of a minimizer $u_0$ of $J_{\alpha} $ in $ \mathfrak L$, and $u_0$ satisfies
the corresponding Euler-Lagrange equation
 \begin{equation}\label{lagrange}
    \alpha P_4 u+6\left(1-\frac{e^{4u}}{\int_{\mathbb{S}^4}e^{4u}dw}\right)= \sum_{i=1}^{5} a_i \xi_i e^{4u} , \quad \hbox{in} \quad  \mathbb{S}^4
    \end{equation}
  for some constants $a_i, i=1, 2, \cdots 5$.

 Equation \eqref{11} can be regarded as the following 4-dimensional counterpart of the constant Gaussian curvature type equation,  or the mean field equation on $ \mathbb{S}^2$,
  \begin{equation}\label{gausian}
    -\alpha \Delta u + \left(1-\frac{e^{2u}}{\int_{ \mathbb{S}^2} {e^{2u}}dw}\right)=0, \quad \xi\in  \mathbb{S}^2.
    \end{equation}
  For \eqref{gausian},  there is a vast literature. See, e.g., \cite{CY87}, \cite{GM} and references therein.

Similar to the prescribing Gaussian curvature equation on $ \mathbb{S}^2$, the Kazdan-Warner obstruction also works for the prescribing $Q$-curvature equation
 \begin{equation*}
    P_4 u+6-Q e^{4u}=0, \quad \xi \in   \mathbb{S}^4.
    \end{equation*}
Indeed, it is shown in \cite[Remarks (3) (ii)  for Cor. 5.4]{Chang95}  and  \cite[Cor. 2.1]{CY87},  by using the invariance of $J_1$ under the conformal transformation,   that  the following Kazdan-Warner condition holds
 \begin{equation}\label{kw}
    \int_{\mathbb{S}^4} \langle\nabla Q, \nabla \xi_i\rangle e^{4u}dw=0,  \quad i=1, 2, \cdots 5.
    \end{equation}

It is an immediate consequence that $a_i=0,  i=1, 2, \cdots 5$  in \eqref{lagrange}, just as in the $ \mathbb{S}^2$ case in \cite{CY87}. (See also \cite{Wei-Xu}, proof of  Theorem 2.6.) Interested reader is referred to  \cite{Chang97,DHL00,Dj08,Hang16,Hang163,Ly19,Lin98,M06,Wei99} for literature  on equations that have conformal structure.
\par
 In what follows, we shall consider axially symmetric functions that are only dependent on $\xi_1$. We shall show that \eqref{11} under axially symmetric setting admits only constant solutions when $\alpha \in [\frac{473 + \sqrt{209329}}{1800}, 1)$. As a consequence we obtain an improved Aubin-type inequality for axially symmetric
functions in $ \mathfrak L$.

  Considering solutions axially symmetric about $\xi_1$-axis and denoting $\xi_1$ by $x$, we can reduce \eqref{11}  to
 \begin{equation}\label{15}
  \alpha \left[  (1-x^2)^2u'\right]'''+6-\frac{8e^{4u}}{\int_{-1}^1(1-x^2)e^{4u}}=0,\quad x\in (-1,1)
 \end{equation}
or equivalently,
    \begin{equation*}
  \alpha  \left[(1-x^2)^3u''\right]''-6  \alpha \left[(1-x^2)^2u'\right]'
  +6(1-x^2)-\frac{8(1-x^2)e^{4u}}{\int_{-1}^1(1-x^2)e^{4u}}=0.
 \end{equation*}
One can refer to Section 2 for the detailed derivation of \eqref{15}. By direct computations, we see that the corresponding functional $ I_\alpha(u)$  can be expressed  as follows
 \begin{equation*}
 \begin{aligned}
     I_\alpha(u)&=\frac{\alpha}{2}\int_{-1}^1 \left((1-x^2)|(1-x^2)u''|^2+6|(1-x^2)u'|^2\right)\\
     &+6\int_{-1}^1 (1-x^2) u-2\ln\left(\frac{3}{4}\int_{-1}^1(1-x^2)e^{4u} \right).
 \end{aligned}
\end{equation*}
Here the function space is $H^2(-1,1)$,  which is the restriction of  $H^2(\mathbb{S}^4)$ in  the set of functions axially symmetric about $\xi_1$-axis and $\xi_1=x$.

The set $\mathfrak L$ is replaced  by
\begin{equation}\label{eq:szego}
\mathfrak L_r=\bigg\{u\in H^2(\mathbb{S}^4):u=u(x) \mbox{ and } \int_{-1}^1 x(1-x^2)e^{4u} =0\bigg\}.
\end{equation}


 Now we state the main results.
  \begin{theorem}\label{main1}
If $\frac{473 + \sqrt{209329}}{1800}\leq \alpha<1$, then  \eqref{15} admits only constant solutions.  As an immediate consequence,  we have
$$\inf_{u\in \mathfrak L_r}I_\alpha(u)=0.
$$
 \end{theorem}

We conjecture that Theorem \ref{main1} holds for  $\frac{1}{2}<\alpha<1$.   Indeed, the lower bound $\frac{473 + \sqrt{209329}}{1800}$ can be improved slightly to $0.5145$ (see discussions in Section 6).  We believe
that $J_{1/2}(u) \ge 0$ for $u\in \mathfrak L $, given the similar inequality for $ \mathbb{S}^2$ as shown in \cite{GM}.

 Now we define the following first momentum functionals on $H^{2}(\mathbb{S}^4)$
\begin{equation*}
\J_\alpha(u)=\frac{\alpha}{2} \int_{\mathbb{S}^4} u (P_4 u) dw+6\int_{\mathbb{S}^4} udw-
\frac{3}{4}\ln \left( (\int_{\mathbb{S}^4}e^{4u}dw)^2 -\sum_{i=1}^5 (\int_{\mathbb{S}^4}e^{4u} \xi_i dw)^2 \right).
\end{equation*}
Note that $\J_{\alpha} (u)= J_{\alpha}(u)$ when $ u \in \mathfrak L$.

As a consequence of Theorem \ref{main1},  we have the following
form of first Szeg\"o limit theorem on $ \mathbb{S}^4$ for axially symmetric functions.

 \begin{theorem}\label{Szego}

$$ \J_{\frac45} (u) \ge 0, \quad \forall u \in \{ u\in H^{2}(\mathbb{S}^4):  u(\xi)=u(\xi_1)\}.
$$
 \end{theorem}


Concerning the classification of axially symmetric solutions at another critical parameter $\alpha=\frac15$, we have the following theorem.
  \begin{theorem}\label{main2}
  If $\alpha=\frac{1}{5}$   and $u$ is an axially symmetric solution to \eqref{11}, then $u$ must be constant.
  \end{theorem}

 Using a bifurcation approach and Theorem \ref{main1}-\ref{main2}, we can also show the existence of nonconstant axially symmetric solution for $\alpha \in (\frac15, \frac12).$

 \begin{theorem}\label{main3}
 There exists a non constant  solution $u_{\alpha}$ to \eqref{15} for $\alpha \in (\frac15, \frac12).$   Moreover, there exists a sequence $\alpha_m \in (\frac15,\frac{473 + \sqrt{209329}}{1800}) $ and a sequence of non constant solutions $u_{\alpha_m}, m=1, 2, \cdots$ to \eqref{15}  such that $\alpha_m \to \frac12$,  $\int_{-1}^1 (1-x^2)e^{4u_{\alpha_m}} =\frac43 $ and  $\|u_{\alpha_m}\|_{L^\infty([-1,1])} \to \infty$ as $m \to \infty$.
 \end{theorem}

We also establish the following proposition concerning the centers of mass and first order momentums of solutions to \eqref{11}.
\begin{proposition}\label{pro31}
  If $u$ solves \eqref{11}, then
  \begin{equation*}
    \int_{\mathbb{S}^4}  e^{4u}  \xi_idw=0 \quad \mbox{ and } \int_{\mathbb{S}^4}  u  \xi_idw=0, \quad i=1,2,\cdots,5,
  \end{equation*}
whenever $\alpha\neq1$.
\end{proposition}

The paper is organized as follows. First, we list some preliminaries and integral identities in Section 2 which will be substantially used in the later context. Section 3 is devoted to the proof of Theorems \ref{main1}-\ref{Szego}. In Section 4, we derive various Pohozaev-type identities and employ them to validate Theorem \ref{main2} together with Proposition \ref{pro31}.
In Section 5,  we carry out a bifurcation analysis of \eqref{15} and its equivalent form,  and prove Theorem \ref{main3}  based on Theorems \ref{main1} and \ref{main2}. The last section is devoted to some discussion  of the improvement of the best constant for $\alpha$.

\vskip4mm
{\section{Preliminaries and Integral Identities}}
 \setcounter{equation}{0}

In this section, we state several important preliminaries and integral identities which will be needed in the proof of Theorem \ref{main1}. We begin by stating some basic facts on spherical geometry of $\mathbb{S}^4$.

Let $\theta_i,\ i=1,2,3,4$ denote the usual angular coordinates on the sphere with
  \begin{equation*}
    \theta_i\in[0, \pi],\ i=1,2,3, \quad  \theta_4\in[0,2\pi]
  \end{equation*}
 and define $x=\xi_1=\cos (\theta_1)$. Then the metric tensor can be given as follows:
 \begin{equation*}
 g_{ij}=
\left( \begin{matrix}
  (1-x^2)^{-1}  & 0 & 0 & 0 \\
  0  & 1-x^2  & 0 & 0 \\
    0 & 0 & (1-x^2)\sin^2\theta_2  & 0 \\
     0 & 0 & 0 & (1-x^2)\sin^2\theta_2 \sin^2\theta_3
\end{matrix}
\right )
 \end{equation*}

 For axially symmetric functions, we have
 \begin{equation*}
   \begin{aligned}
     \int_{\mathbb{S}^4}e^{4u}dw &=\frac{3}{8 \pi^2}\int_{-1}^1\int_0^{\pi}\int_0^{\pi}\int_0^{2\pi} e^{4u}(1-x^2)\sin^2\theta_2 \sin\theta_3 d\theta_4d\theta_3d\theta_2dx\\
     &=\frac{3}{4}\int_{-1}^1(1-x^2)e^{4u}
   \end{aligned}
 \end{equation*}
 and
 \begin{equation*}
    \begin{aligned}
      \Delta u&= |g|^{-\frac{1}{2}}\frac{\partial}{\partial x}\left(|g|^{\frac{1}{2}}g^{11}\frac{\partial u}{\partial x}\right)
     =(1-x^2)^{-1}\frac{\partial}{\partial x}\left[(1-x^2)^2\frac{\partial u}{\partial x}\right]\\
     &=(1-x^2)u''-4xu'.
    \end{aligned}
 \end{equation*}
One further has that
 \begin{equation*}
      \begin{aligned}
        ((1-x^2)\Delta u)''&=-2\Delta u-4x(\Delta u)'+(1-x^2)(\Delta u)''\\
        &=\Delta^2 u-2\Delta u.
      \end{aligned}
 \end{equation*}
Thus, the Paneitz operator on $\mathbb{S}^4$ can be expressed as
 \begin{equation*}
    \begin{aligned}
          P_4 u&=\left[  (1-x^2)^2u''-4x(1-x^2)u'\right]''\\
          &=\left[  (1-x^2)^2u'\right]''',
    \end{aligned}
 \end{equation*}
for $u=u(x)$. Then, we transform the original  equation \eqref{11} on $\mathbb{S}^4$ into an ODE \eqref{15}.


Note that the eigenfunctions associated with Paneitz operator coincide with those associated with the Laplacian. It is natural to introduce Gegenbauer polynomials (see \cite[Chapter 2.4]{MM}), which can be considered as a family of generalized Legendre polynomials.

Let
 \begin{equation*}
  F_k(x)=\frac{(-1)^k\Gamma(2)}{2^k\Gamma(k+2)}\frac{1}{(1-x^2)}\frac{d^k}{dx^k}(1-x^2)^{k+1}
 \end{equation*}
 be the $k$-th  Gegenbauer polynomials. Then $F_k$ satisfies that
 \begin{equation}\label{22}
   (1-x^2)F_k''-4xF_k'+\lambda_k F_k=0,\quad \lambda_k=k(k+3),\ k=0,1,\cdots;
 \end{equation}
 and
 \begin{equation}\label{ortho}
 \int_{-1}^1 (1-x^2) F_k F_l =0 \mbox{ for }k\neq l.
 \end{equation}
Note here $F_k$ is a sphere harmonic of degree $k$. Then it is readily checked that for $x\in(-1,1)$,
 \begin{equation*}
   (1-x^2)\left[  (1-x^2)^2F_k'\right]'''  =(\lambda_k^2+2\lambda_k)(1-x^2)F_k.
 \end{equation*}
Moreover(see \cite{MM,GR}),
 \begin{equation}\label{24}
   |F_k'(x)|\leq \frac{\lambda_k}{4},\quad  \int_{-1}^1 (1-x^2) F_k^2=\frac{8}{(2k+3)(k+1)(k+2)}.
 \end{equation}
We will focus on the gradient of $u$ on the sphere throughout the rest of the paper. Define
\begin{equation}\label{G}
G(x)=(1-x^2)u',
\end{equation}
where $u=u(x)$ is a solution to \eqref{11}. Then we have the following decomposition using the orthogonal polynomials $F_k$'s:
\begin{equation}
\label{decomp}
   G=a_0 F_0+\beta x+a_2\frac{1}{4}(5x^2-1)+\sum_{k=3}^\infty a_kF_k(x).
\end{equation}
Define
\begin{equation*}
G_2=\sum_{k=3}^\infty a_kF_k(x),
\end{equation*}
and
\begin{equation*}
b_k^2=a_k^2\int_{-1}^1(1-x^2)F_k^2,\ k\geq2.
\end{equation*}

We first derive a lemma concerning the constant term $a_0$ in \eqref{decomp}.

\begin{lemma}\label{l31}
  If $u$ is a critical point of $I_\alpha$ whenever $\alpha\neq 1$ , then the function $G(x)$ belongs to $H^2(-1,1)$ and satisfies that  $\int_{-1}^1(1-x^2)G=0$. In other words, $a_0=0$.
\end{lemma}

\begin{proof}
 In view of equation \eqref{15}, we have
   \begin{equation}\label{26}
     \alpha((1-x^2)G)'''+6-\frac{8}{\gamma}e^{4u}=0,
   \end{equation}
where
\begin{equation}\label{gamma}
\gamma=\int_{-1}^1 (1-x^2) e^{4u}.
\end{equation} By differentiating \eqref{26}, we further have
  \begin{equation}
  \label{26I}
     \alpha((1-x^2)G)''''-\frac{32}{\gamma}e^{4u}u'=0.
  \end{equation}
  Multiplying \eqref{26I} by $(1-x^2)^2$ and employing  \eqref{26}, we have
  \begin{equation}\label{27}
   (1-x^2)^2((1-x^2)G)''''-\frac{24}{\alpha}(1-x^2)G-4[(1-x^2)G]'''(1-x^2)G=0.
     \end{equation}
   Integrate \eqref{27} over $[-1,1]$ to get
   \begin{equation}
   \label{272}
     \int_{-1}^1(1-x^2)^2((1-x^2)G)''''-  \frac{24}{\alpha}\int_{-1}^1 (1-x^2)G
     -4\int_{-1}^1((1-x^2)G)'''(1-x^2)G=0.
   \end{equation}
   We use  integration by parts for the first term of \eqref{272}:
    \begin{align}
       \int_{-1}^1 (1-x^2)^2((1-x^2)G)'''' &=\left[ (1-x^2)^2((1-x^2)G)'''\right]\Big|_{-1}^1-  \int_{-1}^1
       \left((1-x^2)^2\right)'((1-x^2)G)'''\nonumber\\
       &=  -\left((1-x^2)^2\right)'((1-x^2)G)''\Big|_{-1}^1+ \int_{-1}^1
       \left((1-x^2)^2\right)''((1-x^2)G)''\nonumber\\
       &=
       \left((1-x^2)^2\right)''((1-x^2)G)'\Big|_{-1}^1-\left((1-x^2)^2\right)'''((1-x^2)G)\Big|_{-1}^1\nonumber\\
       &+\int_{-1}^1
       \left((1-x^2)^2\right)''''(1-x^2)G\nonumber\\
       &=24 \int_{-1}^1(1-x^2)G.\label{273}
   \end{align}
 Similarly, for the last term in \eqref{272},  one has
     \begin{align}
       \int_{-1}^1 ((1-x^2)G)'''(1-x^2)G&=((1-x^2)G)''(1-x^2)G\Big|_{-1}^1-  \int_{-1}^1
       \left((1-x^2)G\right)'((1-x^2)G)''\nonumber\\
       &=   -\frac{1}{2} \left((1-x^2)G'-2xG\right)^2\Big|_{-1}^1=0.\label{274}
     \end{align}
We conclude from the \eqref{272}-\eqref{274} that
   \begin{equation*}
    \left( 24 -\frac{24}{\alpha}\right)\int_{-1}^1(1-x^2)G=0,
   \end{equation*}
which implies that $\int_{-1}^1(1-x^2)G=0$ since $\alpha\neq 1$.
\end{proof}
Next, we state some important integral identities which will be used frequently in the proof of Theorem \ref{main1}.

\begin{lemma}\label{equality}
We establish the following equalities for $G(x)=(1-x^2)u'$ where $u$ is a solution of \eqref{15} and $\alpha>0$.
\begin{equation}\label{220}
  \int_{-1}^1(1-x^2) F_1 G=\frac{4}{15}\beta,
\end{equation}
 \begin{equation}\label{221}
 \int_{-1}^1 (1-x^2)^2\frac{e^{4u}}{\gamma}=\frac{4}{5}(1-\alpha\beta),
 \end{equation}
\begin{equation}\label{222}
  \int_{-1}^1 (1-x^2)  F_k G=-\frac{8}{\alpha(\lambda_k^2+2\lambda_k)}\int_{-1}^1 \frac{e^{4u}}{\gamma}(1-x^2)^2F_k',\quad k\geq 2,
\end{equation}
\begin{equation}\label{223}
     \int_{-1}^1 |[(1-x^2)G]'|^2=\frac{16}{15}\left(5-\frac{1}{\alpha}\right)\beta,
\end{equation}
recalling that $\beta$ is defined in \eqref{decomp} and $\gamma$ is given in \eqref{gamma}.
\end{lemma}
\begin{proof}
  Indeed, we have
\begin{equation*}
  F_0(x)=1,\quad F_1(x)=x,\quad F_2(x)=\frac{1}{4}(5x^2-1).
\end{equation*}
Then it follows from  \eqref{27} that
\begin{equation*}
  \int_{-1}^1 (1-x^2) F_1 G=  \beta \int_{-1}^1 (1-x^2)x^2 =\frac{4}{15}\beta.
\end{equation*}
This finishes the proof of \eqref{220}.

For \eqref{221}, multiplying  \eqref{26}
by $\int_{-1}^x (1-s^2)F_k(s)ds$ with $k\geq 1$
   and integrating over $[-1, 1]$, one has
   \begin{equation}\label{225}
     \int_{-1}^1\int_{-1}^x (1-s^2)F_k(s)\left[\alpha((1-x^2)G)'''+6-\frac{8}{\gamma}e^{4u}\right]dsdx=0.
   \end{equation}
It is easy to see that
\begin{align}
    \int_{-1}^1\int_{-1}^x (1-s^2)F_k(s) ((1-x^2)G)'''dsdx & = \left[((1-x^2)G)''\int_{-1}^x (1-s^2)F_k(s) \right]\Bigg|_{-1}^1\nonumber\\
& - \int_{-1}^1  (1-x^2)F_k  ((1-x^2)G)''\nonumber\\
& = -4G'(1)\int_{-1}^1 (1-s^2)F_k(s) -\left[(1-x^2)F_k(x) ((1-x^2)G)'\right]\Big|_{-1}^1\nonumber\\
& +\int_{-1}^1  ((1-x^2)F_k)' ((1-x^2)G)'\nonumber\\
& = \left[((1-x^2)F_k(x))' (1-x^2)G\right]\Big|_{-1}^1-\int_{-1}^1  ((1-x^2)F_k)'' ((1-x^2)G)\nonumber\\
    &= -\int_{-1}^1  G (1-x^2)\left[(1-x^2)F_k''-4xF_k'-2F_k\right]\nonumber\\
    &=(\lambda_k+2)\int_{-1}^1    (1-x^2)GF_k.\label{218I}
\end{align}
Moreover,
\begin{align}
    \int_{-1}^1\int_{-1}^x  (1-s^2)F_k(s) dsdx &= \left(x\int_{-1}^x (1-s^2)F_k(s)\right)\Big|_{-1}^1- \int_{-1}^1  (1-x^2)xF_k  \nonumber\\
    &= \int_{-1}^1 (1-s^2)F_k(s)-\delta_{1k} \int_{-1}^1  (1-x^2)x^2\nonumber\\
    &=-\frac{4}{15}\delta_{1k}.\label{218II}
\end{align}
We use \eqref{22} to obtain
\begin{equation}\label{226}
  [(1-x^2)^2 F_k']'= (1-x^2)^2 F_k''-4x (1-x^2)  F_k'=-\lambda_k(1-x^2)  F_k
\end{equation}
for $k\geq 1$.

Entailing from \eqref{226}, we have that
 \begin{equation*}
  \int_{-1}^x  (1-s^2)F_k(s)=-\frac{1}{\lambda_k}(1-x^2)^2 F_k'(x).
 \end{equation*}
 Letting $k=1$, we have
\begin{equation}
\label{218III}
    \int_{-1}^1\int_{-1}^x  (1-s^2)F_1(s) ds e^{4u}dx
     =-\frac{1}{4} \int_{-1}^1(1-x^2)^2e^{4u}.
\end{equation}
Keep in mind that
\begin{equation}
\label{218IV}
 (\lambda_k+2)\int_{-1}^1    (1-x^2)GF_k=6\beta \int_{-1}^1    (1-x^2)x^2=\frac{8}{5}\beta.
\end{equation}
 Then, \eqref{221} follows from \eqref{225}-\eqref{218IV}.

 Similarly, letting $k\geq 2$, we conclude  that
 \begin{equation*}
  (\lambda_k+2)\alpha\int_{-1}^1    (1-x^2)GF_k=-\frac{8}{\gamma\lambda_k} \int_{-1}^1(1-x^2)^2 F_k'e^{4u}.
 \end{equation*}
 So \eqref{222} holds.

 For \eqref{223}, multiplying \eqref{27} by $x$ and integrating from $-1$ to $1$, we get
 \begin{equation}\label{227}
   \int_{-1}^1  \left[x (1-x^2)^2((1-x^2)G)''''-\frac{24}{\alpha}x(1-x^2)G-4[(1-x^2)G]'''(1-x^2)xG\right]=0
 \end{equation}
For the first term in \eqref{227}, we have
\begin{align}
    \int_{-1}^1 x (1-x^2)^2((1-x^2)G)'''' &=  \int_{-1}^1   \left(x (1-x^2)^2 \right)''''(1-x^2)G\nonumber\\
    &=120\int_{-1}^1 (1-x^2)xG\nonumber\\
    &=120\beta\int_{-1}^1 (1-x^2)x^2\nonumber\\
    &=32\beta.\label{227I}
\end{align}
For the second term, one has
\begin{align}
     \frac{24}{\alpha}\int_{-1}^1 x(1-x^2)G&=   \frac{24\beta}{\alpha} \int_{-1}^1 (1-x^2)x^2\nonumber\\
    &=\frac{32}{5\alpha}\beta.\label{227II}
\end{align}
For the last term, we find
\begin{align}
4\int_{-1}^1 [(1-x^2)G]'''(1-x^2)xG &= - 4\int_{-1}^1 [(1-x^2)G]''(1-x^2)G- 4\int_{-1}^1x [(1-x^2)G]''[(1-x^2)G]'\nonumber\\
&=4\int_{-1}^1 |[(1-x^2)G]'|^2-2x|[(1-x^2)G]'|^2 \Big|_{-1}^1 +2\int_{-1}^1 |[(1-x^2)G]'|^2\nonumber\\
&=6\int_{-1}^1 |[(1-x^2)G]'|^2. \label{227III}
\end{align}
Therefore, \eqref{223} follows from combining \eqref{227}-\eqref{227III}.
\end{proof}


 \vskip4mm
{\section{Proof of Theorem \ref{main1}}}

  Inspired by \cite{Gui98} and \cite{Gui-Wei}, our basic strategy is to assume $\beta\neq 0$, and show that  it leads to a contradiction with the range of $\alpha$. It is fairly easy to see from \eqref{223} that if $\beta=0$, then $\nabla u=0$, which shows that $u$ is a constant. One important new ingredient is the surprising a priori  estimate in Lemma \ref{132} regarding the derivative of the gradient of $u$.


We now give the key estimate on the derivative of $G$, which is defined in \eqref{G}. Note that the lemma is true for general $\alpha>0$.
\begin{lemma}\label{132}
Let $M=\max_{x\in[-1,1]}G'(x)$.
Then we have
\begin{equation}\label{21}
  M\leq\frac{1}{\alpha},
\end{equation}
i.e.,
\begin{equation*}
G'(x)\leq\frac{1}{\alpha}
\end{equation*}
for all $x\in[-1,1]$.
\end{lemma}
\begin{proof}
Take $M=G'(x_0)$ for some $x_0\in[-1,1]$. We  first prove the lemma if  $x_0\in(-1,1)$.
    After some calculations, \eqref{26} becomes
\begin{equation}\label{28}
    \alpha[-6G'-6xG''+(1-x^2)G''']+6-\frac{8}{\gamma}e^{4u}=0, \quad x\in(-1,1).
\end{equation}
At $x=x_0$, we have
\begin{equation*}
   G'(x_0)=M,\quad G''(x_0)=0\quad\mbox{and }\quad   G'''(x_0)\leq0.
\end{equation*}
Consequently, $M\leq\frac{1}{\alpha}$.

If $x_0 \not \in(-1,1)$,  then we may assume without loss of generality that
\begin{equation*}
  \sup_{x\in(-1,1)}G'(x)=\lim_{x_n\rightarrow 1}G'(x_n)
\end{equation*}
for some sequence $\{x_n\} \subset (-1, 1)$.
Let  $r=|x'|=\sqrt{1-x^2}$,  then we write
\begin{equation*}
  G(x)=\bar G(r) \mbox{ and }u(x)=\bar u(r) \quad \mbox{ for } r\in[0,1)\mbox{ and }x\in(0,1].
\end{equation*}
 It is well known that $\bar u(r)$ can be extended evenly and $\bar u(r)\in \mathcal{C}^\infty(-\frac{1}{2},\frac{1}{2})$. For $r\in(0,\frac{1}{2})$,
 \begin{equation*}
   u'(x)=\bar u_r(r)\frac{dr}{dx}=-\frac{\bar u_r(r)}{r}\sqrt{1-r^2},
 \end{equation*}
then, one has
\begin{equation}\label{29}
  G(x)=(1-x^2)u'(x)=-r\bar u_r(r)\sqrt{1-r^2}=\bar G(r),\quad r\in(0,\frac{1}{2}).
\end{equation}
A direct calculation shows that
\begin{equation*}
  G'(x)=\bar G_r(r)\frac{dr}{dx}=\left(\bar u_{rr}(r)+\frac{\bar u_{r}(r)}{r}\right)(1-r^2)-r\bar u_{r}(r).
\end{equation*}
It is easy to see that $\lim_{x\rightarrow 1}G'(x)=2\bar u_{rr}(0).$  Note that
\begin{equation*}
  \lim_{r\rightarrow0}\bar u_{r}(r)=\bar u_{r}(0)=0\mbox{ and } \lim_{r\rightarrow0}\frac{\bar u_{r}}{r}=\bar u_{rr}(0).
 \end{equation*}
We can write
\begin{equation}\label{210}
  \bar u_r(r)=t_1 r+t_2r^3+O(r^5) \quad \mbox{near }r=0.
\end{equation}
Then
\begin{equation}\label{211}
 G'(1)=\lim_{x\rightarrow 1}G'(x)=2t_1=\sup_{x\in(-1,1)}G'(x)>0.
\end{equation}
The last inequality is insured by the fact that \eqref{29} implies
\begin{equation*}
  G(1)=G(-1)=\lim_{x\rightarrow \pm1}G(x)=0.
\end{equation*}
Furthermore, by \eqref{210},
\begin{equation*}
  G'(x)=2t_1+(4t_2-3t_1)r^2+O(r^4) \quad \mbox{near }r=0.
\end{equation*}
It follows from  \eqref{211} that
\begin{equation*}
4t_2-3t_1\leq 0.
\end{equation*}
By similar arguments, we obtain that  near $r=0$,
\begin{equation*}
  \begin{aligned}
    G''(x)&=\frac{dG'(x)}{dr}\frac{dr}{dx}\\
    &=-2(4t_2-3t_1+O(r^2))\sqrt{1-r^2}.
  \end{aligned}
\end{equation*}
Therefore,
\begin{equation}\label{212}
 \lim_{x\rightarrow1}-xG''(x)\leq 0.
\end{equation}
By similar calculations again, near $r=0$,
\begin{equation*}
  (1-x^2)G'''(x)=r^2\frac{dG''(x)}{dr}\frac{dr}{dx}=O(r^2).
\end{equation*}
This ensures that
\begin{equation}\label{213}
 \lim_{x\rightarrow1} (1-x^2)G'''(x)=0.
\end{equation}
Using \eqref{28} together with \eqref{211}-\eqref{213}, we have
\begin{equation}\label{214}
G'(1)= \sup_{x\in(-1,1)}G'(x)\leq \frac{1}{6\alpha}\left(6-\frac{8}{\gamma}e^{4u}\right)\leq\frac{1}{\alpha}.
\end{equation}
\end{proof}

\begin{remark}\label{optimal}
When $\alpha=1$,  there is a family of solutions $ u=-\ln(1-ax)$ to \eqref{15}  for
any $a \in (0, 1)$.  Straightforward computations show that the estimate in Lemma \ref{132} is indeed optimal in general.   However, given some extra information,  the estimate may be improved slightly (see the discussion in Section 6 below for details).
\end{remark}

\begin{lemma}
\label{l33}
Concerning a semi-norm of $G$, we have the following estimate:
\begin{equation}\label{keyestimate}
  \lfloor G\rfloor^2\leq \left(\frac{4}{\alpha}-6\right)\int_{-1}^1 |[(1-x^2)G]'|^2+\frac{16}{\alpha}\int_{-1}^1 (1-x^2)G^2,
\end{equation}
where
\begin{equation}\label{norm}
\lfloor G\rfloor^2= \frac43\int_{\mathbb{S}^4} (P_4 G ) G dw= \int_{-1}^1 (1-x^2)\left[(1-x^2)^2G'\right]'''G.
\end{equation}
\end{lemma}

\begin{proof}
Multiplying \eqref{27} by $G$ and integrating over $[-1,1]$, we have
   \begin{equation}\label{215}
     \int_{-1}^1((1-x^2)G)'''' (1-x^2)^2 G -  \frac{24}{\alpha}\int_{-1}^1 (1-x^2)G^2
     -4\int_{-1}^1((1-x^2)G)'''(1-x^2)G^2=0.
   \end{equation}
For the first term in \eqref{215}, after   integration by parts, one has
\begin{equation*}
     \begin{aligned}
       \int_{-1}^1((1-x^2)G)''''  (1-x^2)^2G &= (1-x^2)^2 G((1-x^2)G)'''\Big|_{-1}^1-  \int_{-1}^1
       \left((1-x^2)^2G\right)'((1-x^2)G)'''\\
       &=  -  \int_{-1}^1
        ((1-x^2)G)''' \left[(1-x^2)^2G'-4x(1-x^2)G\right]\\
    &=\int_{-1}^1 (1-x^2)\left[(1-x^2)^2G'\right]'''G +4 \int_{-1}^1
        \left[((1-x^2)G)'''x(1-x^2)G\right]
     \end{aligned}
\end{equation*}
and
\begin{equation*}
     \begin{aligned}
      \int_{-1}^1((1-x^2)G)''' x(1-x^2)G &= -\int_{-1}^1((1-x^2)G)''\left[(1-x^2)G+x((1-x^2)G)'\right]\\
      &=-\frac{x}{2}|((1-x^2)G)'|^2\Big|_{-1}^1+\int_{-1}^1|((1-x^2)G)'|^2+\frac{1}{2}\int_{-1}^1|((1-x^2)G)'|^2\\
      &=\frac{3}{2}\int_{-1}^1|((1-x^2)G)'|^2.
     \end{aligned}
   \end{equation*}
It follows that
   \begin{equation}
   \label{lm321}
     \int_{-1}^1((1-x^2)G)''''  (1-x^2)^2G=\int_{-1}^1 (1-x^2)\left((1-x^2)^2G'\right)'''G +6\int_{-1}^1|((1-x^2)G)'|^2.
   \end{equation}
For the last term in \eqref{215}, we obtain
   \begin{align}
        &\quad\int_{-1}^1((1-x^2)G)'''(1-x^2)G^2= \int_{-1}^1[-6G'-6xG''+(1-x^2)G'''](1-x^2)G^2\nonumber\\
        &=\int_{-1}^1[-6G'+(1-x^2)G'''](1-x^2)G^2 - 3\int_{-1}^1(1-x^2)^2GG'G''-\frac{3}{2}\int_{-1}^1
         (1-x^2)^2G^2G'''\nonumber\\
        &=-2\int_{-1}^1(1-x^2)(G^3)'-\frac{1}{2}\int_{-1}^1(1-x^2)^2[G^2G'''+6GG'G'']\nonumber\\
         &=-2\int_{-1}^1(1-x^2)(G^3)'-\frac{1}{2}\int_{-1}^1(1-x^2)^2\left[\frac{1}{3}(G^3)'''-2(G')^3\right]\nonumber\\
         &=-4\int_{-1}^1xG^3+\frac{1}{6}\int_{-1}^1[(1-x^2)^2]'''G^3+\int_{-1}^1(1-x^2)^2(G')^3\nonumber\\
         &=\int_{-1}^1(1-x^2)^2(G')^3.\label{lm322}
   \end{align}
By \eqref{lm321} and \eqref{lm322},  \eqref{215} is equivalent to
\begin{align}
    &\quad\int_{-1}^1 (1-x^2)\left((1-x^2)^2G'\right)'''G +6\int_{-1}^1|((1-x^2)G)'|^2\nonumber\\
    & -  \frac{24}{\alpha}\int_{-1}^1 (1-x^2)G^2
     - 4\int_{-1}^1(1-x^2)^2(G')^3=0.\label{216}
\end{align}
Note that
\begin{align}
  &\quad 6\int_{-1}^1|((1-x^2)G)'|^2-4\int_{-1}^1(1-x^2)^2(G')^3\nonumber\\
  &=6\int_{-1}^1(1-x^2)^2(G')^2\left[1-\frac{2}{3}G'\right]+12\int_{-1}^1(1-x^2)G^2.\label{217}
\end{align}
We deduce from  \eqref{21}, \eqref{216} and \eqref{217} that
\begin{equation}\label{218}
   \begin{aligned}
    \left(\frac{24}{\alpha}-12\right)&\int_{-1}^1 (1-x^2)G^2\geq  \lfloor G\rfloor^2+ \left(6-\frac{4}{\alpha}\right)\int_{-1}^1(1-x^2)^2(G')^2.
\end{aligned}
\end{equation}
It is easy to see
\begin{equation}\label{219}
  \int_{-1}^1 |[(1-x^2)G]'|^2-\int_{-1}^1(1-x^2)^2(G')^2=2\int_{-1}^1(1-x^2)G^2.
\end{equation}
 Then \eqref{218} and \eqref{219} lead to that
 \begin{equation}\label{key}
     \lfloor G\rfloor^2\leq \left(\frac{4}{\alpha}-6\right)\int_{-1}^1 |[(1-x^2)G]'|^2+\frac{16}{\alpha}\int_{-1}^1 (1-x^2)G^2.
 \end{equation}
\end{proof}

\begin{corollary}\label{cor3.3}
If $\frac{2}{3}<\alpha<1$, any axially symmetric solution to \eqref{11} must be constant.
\end{corollary}
\begin{proof}
Using the facts  that the first eigenvalue of Laplacian on $\mathbb{S}^4$ is $\lambda_1=4$ as in \eqref{22} and the first eigenvalue of $P_4$ is $\lambda_1(\lambda_1+2)=24$,  we obtain from Lemma \ref{l33} immediately that when $\alpha>2/3$, $G$ must be constant $0$ and hence $u$ must be  constant.
\end{proof}
\begin{proof}[Proof of Theorem \ref{main1}]
We shall use higher order  eigenfunctions in \eqref{22} to gain better estimate for $\alpha$ and prove the main theorem. We first define the following quantity
\begin{equation}\label{224}
  D:=\sum_{k=3}^\infty\left[\lambda_k(\lambda_k+2)-\left(10-\frac{4}{3\alpha}\right)(\lambda_k+2)-\frac{16}{\alpha}
    \right]b_k^2.
\end{equation}
From \eqref{223}, Lemma \ref{l33} and the definition of $G_2$, it is easy to see that
\begin{equation}\label{2241}
\begin{aligned}
  D &= \int_{-1}^1 (1-x^2)\left((1-x^2)^2G_2'\right)'''G_2
   -\left(10-\frac{4}{3\alpha}\right)\int_{-1}^1|((1-x^2)G_2)'|^2-
    \frac{16}{\alpha}\int_{-1}^1 (1-x^2)G_2^2 \\
    &\leq \lfloor G\rfloor^2 -\left(10-\frac{4}{3\alpha}\right)\int_{-1}^1 |[(1-x^2)G]'|^2-\frac{16}{\alpha}\int_{-1}^1 (1-x^2)G^2+\left(36+\frac{8}{\alpha}\right)\beta^2\int_{-1}^1(1-x^2)F_1^2\\
    &\leq \left(\frac{16}{3\alpha}-16\right)\int_{-1}^1 |[(1-x^2)G]'|^2+\left(36+\frac{8}{\alpha}\right)\frac{4\beta^2}{15} \\
    &\leq \frac{16\beta}{15}\left[\left(9+\frac{2}{\alpha}\right)\beta
    +\left(\frac{16}{3\alpha}-16\right)\left(5-\frac{1}{\alpha}\right)\right].
\end{aligned}
\end{equation}


In what follows, we assume that $\beta\neq 0$. From \eqref{220} and \eqref{221}, one has
\begin{equation}\label{228}
0<\beta<\frac{1}{\alpha}.
\end{equation}
By  \eqref{24} and \eqref{222},
\begin{equation*}
 \begin{aligned}
  b_k^2&=a_k^2\int_{-1}^1(1-x^2)F_k^2=\frac{1}{\int_{-1}^1(1-x^2)F_k^2}\left[\frac{8}{\alpha(\lambda_k^2+2\lambda_k)}\int_{-1}^1 \frac{e^{4u}}{\gamma}(1-x^2)^2F_k'\right]^2\\
  &\leq \frac{(2k+3)(k+1)(k+2)}{8}\left[\frac{8}{\alpha(\lambda_k^2+2\lambda_k)}\frac{\lambda_k}{4}
  \frac{4}{5}(1-\alpha\beta)\right]^2.\\
 \end{aligned}
\end{equation*}
Hence one has
\begin{equation}\label{229}
  b_k^2\leq\frac{8(2k+3)}{25(\lambda_k+2)}\left(\frac{1}{\alpha}-\beta\right)^2,\quad k\geq 2.
\end{equation}
 In particular, we obtain
 \begin{equation}\label{230}
   \frac{5}{7}|a_2|\leq\left(\frac{1}{\alpha}-\beta\right).
 \end{equation}
It follows from  $\alpha>\frac{1}{2}$ and \eqref{224}-\eqref{228} that
\begin{equation}\label{231}
\beta\geq\frac{16}{13}\left(1-\frac{1}{3\alpha}\right)\left(5-\frac{1}{\alpha}\right)\geq \frac{16}{13}
\end{equation}
and
\begin{equation*}
 \left(9+\frac{2}{\alpha}\right)\frac{1}{\alpha}
    +\left(\frac{16}{3\alpha}-16\right)\left(5-\frac{1}{\alpha}\right)\geq 0,
\end{equation*}
which indicates that
\begin{equation*}
  \alpha< 0.5732.
\end{equation*}

From \eqref{223}, \eqref{224} and \eqref{230},  we derive that
\begin{equation*}
\begin{aligned}
  &\quad \frac{16\beta}{15}\left[\left(9+\frac{2}{\alpha}\right)\beta
    +\left(\frac{16}{3\alpha}-16\right)\left(5-\frac{1}{\alpha}\right)\right]\\
    &\geq D\geq
    \left[\lambda_3-\left(10-\frac{4}{3\alpha}\right)-\frac{16}{\alpha(\lambda_3+2)} \right]\sum_{k=3}^\infty(\lambda_k+2)b_k^2\\
    &=8\left(1+\frac{1}{15\alpha}\right)\int_{-1}^1|[(1-x^2)G_2]'|^2\\
    &=8\left(1+\frac{1}{15\alpha}\right)\left[\int_{-1}^1|[(1-x^2)G]'|^2-\frac{8}{5}\beta^2-\frac{8}{7}a_2^2\right]\\
    &\geq 8\left(1+\frac{1}{15\alpha}\right)
    \left[\frac{16\beta}{15}\left(5-\frac{1}{\alpha}\right)-\frac{8}{5}\beta^2
    -\frac{56}{25}\left(\frac{1}{\alpha}-\beta\right)^2\right].
\end{aligned}
 \end{equation*}
Then one has
\begin{equation*}
\begin{aligned}
 &\quad 2\beta\left(\frac{16}{3\alpha}-16\right)\left(5-\frac{1}{\alpha}\right)-\frac{16\beta}{15}\left(5-\frac{1}{\alpha}\right)
 \left(15+\frac{1}{\alpha}\right)\\
 &\geq-2\left[\left(9+\frac{2}{\alpha}\right)+\frac{4}{5}\left(15+\frac{1}{\alpha}\right)\right]\beta^2
    -\frac{56}{25} \left(15+\frac{1}{\alpha}\right)\left(\frac{1}{\alpha}-\beta\right)^2.
 \end{aligned}
\end{equation*}
After some straightforward computations, we obtain
\begin{equation}\label{233}
\begin{aligned}
  & \quad 2\beta\left[\left(5-\frac{1}{\alpha}\right)
    \left(\frac{24}{5\alpha}-24\right)+\frac{1}{\alpha}\left(\frac{14}{5\alpha}+21\right) \right]\\
    &\geq \left(\frac{1}{\alpha}-\beta\right)\left[2\beta\left(\frac{14}{5\alpha}+21\right)
    -\frac{56}{25} \left(15+\frac{1}{\alpha}\right)\left(\frac{1}{\alpha}-\beta\right)\right]\\
    &\geq0,
\end{aligned}
\end{equation}
 From \eqref{228} and \eqref{231}, we conclude that
 \begin{equation}\label{234}
   0\leq \left(5-\frac{1}{\alpha}\right)
    \left(\frac{24}{5\alpha}-24\right)+\frac{1}{\alpha}\left(\frac{14}{5\alpha}+21\right)\leq10
 \end{equation}
The first inequality suggests that
\begin{equation*}
  \alpha<0.5444.
\end{equation*}
Furthermore,  it follows from \eqref{231}-\eqref{234}  that
\begin{equation}\label{236}
\begin{aligned}
  \frac{1}{\alpha}-\beta&\leq 20 \left(\frac{2002}{25}\beta-\frac{10136}{325\alpha} \right)^{-1}\beta\\
  &\leq\frac{325}{588}\beta.
  \end{aligned}
\end{equation}

Next, we fix an integer $n\geq 3$. After some computations,    we get
\begin{equation*}
  \begin{aligned}
&\quad\sum_{k=3}^n\left(\lambda_k-\lambda_{n+1}-\frac{4}{15\alpha}\right) (\lambda_k+2)b_k^2+\left(\lambda_{n+1}- 10+\frac{4}{5\alpha}\right)\left[\frac{16\beta}{15}\left(5-\frac{1}{\alpha}\right)-\frac{8}{5}\beta^2
    -\frac{8}{7}a_2^2\right]\\
    &\leq\sum_{k=3}^\infty\left[\lambda_k(\lambda_k+2)-\left(10-\frac{4}{3\alpha}\right)(\lambda_k+2)-\frac{16}{\alpha}
    \right]b_k^2\\
    &\leq \frac{16\beta}{15}\left[\left(9+\frac{2}{\alpha}\right)\beta
    +\left(\frac{16}{3\alpha}-16\right)\left(5-\frac{1}{\alpha}\right)\right].
  \end{aligned}
\end{equation*}
Hence, we have
\begin{equation*}
  \begin{aligned}
0&\leq\frac{16\beta}{15}\left(5-\frac{1}{\alpha}\right)\left(\frac{68}{15\alpha}-6-\lambda_{n+1} \right)+\frac{8}{15}\beta^2
\left(3\lambda_{n+1}-12+\frac{32}{5\alpha} \right) \\
&+\frac{8}{7} \left(\lambda_{n+1}- 10+\frac{4}{5\alpha}\right)a_2^2+\sum_{k=3}^n\left(\lambda_{n+1}-\lambda_k+\frac{4}{15\alpha}\right) (\lambda_k+2)b_k^2.
  \end{aligned}
\end{equation*}
By \eqref{229}, one further has
\begin{equation}\label{239}
\begin{aligned}
0&\leq\frac{16\beta}{15}\left(5-\frac{1}{\alpha}\right)\left(\frac{68}{15\alpha}-6-\lambda_{n+1} \right)+\frac{8}{15}\beta^2
\left(3\lambda_{n+1}-12+\frac{32}{5\alpha} \right) \\
&+\frac{8}{25}\left(\frac{1}{\alpha}-\beta\right)^2\left[7\left(\lambda_{n+1}- 10+\frac{4}{5\alpha}\right)+c_n\right],
  \end{aligned}
\end{equation}
where $$c_n:=\sum_{k=3}^n\left(\lambda_{n+1}-\lambda_k+\frac{4}{15\alpha}\right)(2k+3).$$
A straightforward calculation shows
\begin{equation*}
  c_n=\frac{1}{2}\lambda_{n+1}^2+\left(\frac{4}{15\alpha}-14\right)\lambda_{n+1}+90-(n+16)\frac{4}{15\alpha}.
\end{equation*}
Therefore,    \eqref{239}  is equivalent to
\begin{equation*}
\begin{aligned}
   0&\leq10\beta\left(5-\frac{1}{\alpha}\right)\left(\frac{68}{15\alpha}-6-\lambda_{n+1} \right)+5\beta^2
\left(3\lambda_{n+1}-12+\frac{32}{5\alpha} \right)+3\bar c_n\left(\frac{1}{\alpha}-\beta\right)^2.
\end{aligned}
\end{equation*}
Here
\begin{equation*}
 \bar c_n=\frac{1}{2}\lambda_{n+1}^2-7\lambda_{n+1}+20+\frac{4}{15\alpha}\left(\lambda_{n+1}+5-n\right).
\end{equation*}
We use the same technique as in \eqref{233} to obtain
\begin{equation}\label{242}
\begin{aligned}
    &\quad \frac{5\beta}{3}\left[-\frac{8}{\alpha^2}+ (15\lambda_{n+1}+136)\frac{1}{\alpha}-180-30\lambda_{n+1}\right]\\
&\geq\left(\frac{1}{\alpha}-\beta\right)\left[5\beta\left(3\lambda_{n+1}-12+\frac{32}{5\alpha} \right)-3\bar c_n\left(\frac{1}{\alpha}-\beta\right)\right].
\end{aligned}
\end{equation}
When $n=3$, we derive from \eqref{236} that
\begin{equation*}
\begin{aligned}
  0&<\left(\frac{1}{\alpha}-\beta\right)\left(\beta+18\frac{\beta}{\alpha}\right)\\
  &\leq \mbox{RHS  of}  \, \eqref{242}.
\end{aligned}
\end{equation*}
 A direct   calculation  suggests that
 \begin{equation}\label{243}
   \alpha< \frac{139 + \sqrt{17281}}{510}:=\alpha_3.
 \end{equation}

Next, we consider $\alpha\in[\alpha_{n+1},\alpha_n)$ with $f_{n}(\alpha_n)=0$ and $\alpha_n\in\left(\frac{1}{2},1\right)$, where
\begin{equation}\label{fn}
   f_{n}(\alpha)=-\frac{8}{\alpha^2}+136\frac{1}{\alpha}-180-15\lambda_{n+1}\left(2-\frac{1}{\alpha}\right).
\end{equation}
 It is readily checked that
\begin{equation*}
  f_{n}(\alpha_{n+1})>0\quad\mbox{ and }\quad  f_{n+1}(\alpha_{n})<0.
\end{equation*}
We now claim that  there exists some $d_n>0$ for $n=3$  or $4$, such that for $\alpha\in[\alpha_{n+1},\alpha_n)$,
 \begin{equation}\label{assert}
 \begin{cases}
   \frac{1}{\alpha}-\beta\leq \frac{d_n}{\lambda_{n+2}};\\
    15(\lambda_{n+2}-4)\beta+\frac{32\beta}{\alpha}-\frac{3\bar c_{n+1}d_n}{\lambda_{n+2}}>0.
 \end{cases}
 \end{equation}
Note that $f_{n+1}(\alpha)\leq f_{n+1}(\alpha_{n+1})= 0$  when $\alpha\in[\alpha_{n+1},\alpha_n)$. We see from \eqref{242} and \eqref{assert} that
\begin{equation}\label{contra}
  0\geq\frac{5\beta}{3} f_{n+1}(\alpha)\geq\left(\frac{1}{\alpha}-\beta\right)\left[ 15\beta(\lambda_{n+2}-4)+\frac{32\beta}{\alpha}-3\bar c_{n+1}\left(\frac{1}{\alpha}-\beta\right) \right]>0.
\end{equation}
There is a contradiction.

 We are ready to prove the assertion  \eqref{assert}. First, we  study more accurately  for  the bound in \eqref{236} when $\alpha\in[\alpha_{n+1},\alpha_n)$. Let
\begin{equation*}
   h(\alpha)=\frac{16}{13}\left(1-\frac{1}{3\alpha}\right)\left(5-\frac{1}{\alpha}\right)
\end{equation*}
and
\begin{equation*}
  \bar h(\alpha)=\left(5-\frac{1}{\alpha}\right)
    \left(\frac{24}{5\alpha}-24\right)+\frac{1}{\alpha}\left(\frac{14}{5\alpha}+21\right).
\end{equation*}
From \eqref{231}-\eqref{234}, it follows   that
\begin{equation}\label{upbd}
  \begin{aligned}
   \frac{1}{\alpha}-\beta&\leq \left[2\beta\left(\frac{98}{25\alpha}+\frac{189}{5}\right)
    -\frac{56}{25} \left(15+\frac{1}{\alpha}\right) \frac{1}{\alpha}\right]^{-1}2\bar h(\alpha_{n+1})\beta\\
    &\leq\left[2 h(\alpha_{n+1})\left(\frac{98}{25\alpha_n}+\frac{189}{5}\right)
    -\frac{56}{25} \left(15+\frac{1}{\alpha_{n+1}}\right) \frac{1}{\alpha_{n+1}}\right]^{-1}2\bar h(\alpha_{n+1})\beta\\
    &:=\gamma_n \beta.
  \end{aligned}
\end{equation}
Then we derive from \eqref{242}, \eqref{fn} and  \eqref{upbd}   that
\begin{equation}\label{recur}
   \begin{aligned}
\frac{5\beta}{3}f_{n}(\alpha)&\geq\left(\frac{1}{\alpha}-\beta\right)\left[15\beta(\lambda_{n+1}-4)+\frac{32\beta}{\alpha} -3\bar c_n\left(\frac{1}{\alpha}-\beta\right)\right]\\
&\geq \beta \left(\frac{1}{\alpha}-\beta\right)\left[15(\lambda_{n+1}-4)+\frac{32}{\alpha} -3\bar c_n\gamma_n\right]\\
&\geq  \omega_n\beta \left(\frac{1}{\alpha}-\beta\right),
   \end{aligned}
\end{equation}
 where
 \begin{equation*}
 \begin{aligned}
  \omega_n&=15(\lambda_{n+1}-4)+\frac{32}{\alpha_{n}}-\gamma_n
\left(\frac32\lambda_{n+1}^2-21\lambda_{n+1}+60+\frac{4}{5 \alpha_{n+1}}(\lambda_{n+1}+5-n)  \right)\\
&:=A_n-B_n\gamma_n.
\end{aligned}
 \end{equation*}
Thus,
\begin{equation}\label{assert1}
   \frac{1}{\alpha}-\beta\leq \frac{5f_{n}(\alpha_{n+1})}{3\omega_n}.
\end{equation}
One can use \eqref{assert1} to prove the first claim in \eqref{assert}. Then the other one is ensured by some calculations.

\par
More precisely, if $n=3$,  then $\alpha\in[\alpha_4,\alpha_3)$.
After some computations,  we find $\gamma_3<0.186$ and so
\begin{equation}\label{n=3-1}
  \frac{1}{\alpha}-\beta\leq \frac{25.553*0.137}{\lambda_5}\leq \frac{3.51}{\lambda_5}.
\end{equation}
Furthermore,
\begin{equation}\label{n=3-2}
 A_4- \frac{3.51}{\lambda_5h(\alpha_4)}B_4\ge 600.3-1682.9*0.0064\geq 589.5>0.
\end{equation}
Combining \eqref{contra} and \eqref{n=3-2}, we find
\begin{equation}\label{n=3contra}
\begin{aligned}
      0&=\frac{5\beta}{3} f_{4}(\alpha_4)\geq \frac{5\beta}{3} f_{4}(\alpha)\geq\left(\frac{1}{\alpha}-\beta\right)\left[  A_4- \frac{3.51}{\lambda_5h(\alpha_4)}B_4\right]>0,
\end{aligned}
\end{equation}
which yields that $\alpha<\alpha_4$.
\par
Similarly, if $n=4$, then  $\alpha\in[\alpha_5,\alpha_4)$. Here
\begin{equation*}
  \alpha_5=\frac{473 + \sqrt{209329}}{1800}.
\end{equation*}
 One has $\gamma_4<0.249$, so
\begin{equation*}
  \frac{1}{\alpha}-\beta\leq \frac{23.0*0.298}{\lambda_6}\leq \frac{6.855}{\lambda_6}
\end{equation*}
and
\begin{equation*}
 A_5- \frac{6.855}{\lambda_6h(\alpha_5)}B_5\ge 811.27-3383.58*0.095\geq 489.5>0.
\end{equation*}
 The previous arguments show  that $\alpha<\alpha_5$.
 This proves Theorem \ref{main1} with
   $ \alpha \in (\frac{473 + \sqrt{209329}}{1800}\approx0.51695, 1).$
   The range of $\alpha$ for Theorem \ref{main1} to hold can be slightly improved to $0.5145\le \alpha<1$, see Section 6 for discussions.
\end{proof}
\begin{remark}
  The approach used  in the case  $n=3$ or $4$ for \eqref{assert} does not work for $n\geq 5$.   The main obstacle is that $\bar c_n$ contains a term involving   $-\lambda_{n+1}^2$,   so we can not guarantee that    the  value of $\omega_n$ is positive for $n\geq 5$.  Let us take  $n=5$ as an example.  Some computations indicate that $\gamma_5\leq 0.2994$ and then
\begin{equation*}
  \omega_5= A_5-\gamma_5 B_5 \approx-632,
\end{equation*}
 which shows that there does not exist such a $d_5$   that the assertion  \eqref{assert}  holds for $n=5$.  Therefore, it seems impossible to get a contradiction similar to \eqref{n=3contra}.
\end{remark}

Next we shall show  Theorem \ref{Szego} as an immediate consequence of Theorem \ref{main1} and invariance of $\J_{\frac45}$ under a family of conformal transformations
$\phi_{P, t}, P \in \mathbb{S}^4, t>0$  of $\mathbb{S}^4$.

Following \cite{Chang95}, we define $\phi_{P, t}, P \in \mathbb{S}^4, t>0$ to be $ \phi_{P, t}(\xi)=\tilde \xi:= \pi_{P}^{-1}(ty)$, where $y=\pi_{P}(\xi)$ is the stereographic project of $\mathbb{S}^4$  from $P$ as the north pole to the equatorial plane.
In particular, we denote $\phi_{t}=\phi_{P_0, t}$ where $P_0=(1, 0, \cdots 0) $.

Given $u \in H^2(\mathbb{S}^4)$ and $t>0$,  let
$$
v(\xi)= u(\phi_t(\xi)) + \frac{5}{4^2} \ln |det(d \phi_t)|, \quad \xi \in \mathbb{S}^4.
$$

We have the following invariance property of $\J_{\frac45}$ under  the transformation
$: u \to v$.
\begin{proposition}\label{invariance}
$$
\J_{\frac45}(u)=\J_{\frac45}(v), \quad  \forall u \in H^2 (\mathbb{S}^4),  \, \, t>0.
$$
\end{proposition}

\begin{proof}
The invariance of
$$ \frac{2}{5} \int_{\mathbb{S}^4} u (P_4 u) dw+6\int_{\mathbb{S}^4} udw
$$
can be proven similarly  as part (a) of  Theorem 4.1 in \cite{Chang95}.
We only need to check that
\begin{equation}\label{invariance-expon}
(\int_{\mathbb{S}^4}e^{4v}dw)^2 -\sum_{i=1}^5 (\int_{\mathbb{S}^4}e^{4v} \xi_i dw)^2=
(\int_{\mathbb{S}^4}e^{4u}dw)^2- \sum_{i=1}^5(\int_{\mathbb{S}^4}e^{4u} {\tilde \xi}_i dw)^2.
\end{equation}
Indeed, after a proper rotation, we may assume that $P=P_0$.  Letting $ a= \frac{1-t^2}{1+t^2} $,  we have
$$
 \xi_1= \frac{a+{\tilde \xi}_1} { 1+ a {\tilde \xi}_1}, \, \quad  \xi_i= \frac{\sqrt{1-a^2} {\tilde \xi}_i} { 1+ a {\tilde \xi}_1}, \quad i=2, 3, 4, 5
$$
and
$$
|det(d\phi_t)|^{\frac14}(\xi)=\frac{\sqrt{1-a^2} }{1-a \xi_1}= |det(d\phi_t^{-1})|^{-\frac14}({\tilde \xi})=
\frac{1+a {\tilde \xi}_1} { \sqrt{1-a^2} }.
 $$

Hence
\begin{align*}
& (\int_{\mathbb{S}^4}e^{4v}dw)^2 - (\int_{\mathbb{S}^4}e^{4v} \xi_1 dw)^2\\
=&\left(\int_{\mathbb{S}^4}e^{4u(\phi_t(\xi))} |det( d\phi_t)(\xi)|^{1+\frac14} dw \right)^2 - \left(\int_{\mathbb{S}^4}e^{4u(\phi_t(\xi)}  |det (d\phi_t)(\xi)|^{1+\frac14} \xi_1   dw \right)^2\\
= &\left(\int_{\mathbb{S}^4}e^{4u} |det( d\phi_t^{-1} )|^{-\frac14}({\tilde \xi}) dw\right)^2 - \left(\int_{\mathbb{S}^4}e^{4u}  |det (d\phi_t^{-1})|^{-\frac14} ({\tilde \xi})  \xi_1  dw \right)^2  \\
=&\left(\int_{\mathbb{S}^4}e^{4u} |det( d\phi_t^{-1} )|^{-\frac14}({\tilde \xi}) (1 - \xi_1 ) dw \right) \times \left(\int_{\mathbb{S}^4}e^{4u} |det( d\phi_t^{-1} )|^{-\frac14}({\tilde \xi}) (1+ \xi_1 ) dw \right)\\
=&\left(\int_{\mathbb{S}^4}e^{4u}  (1 - {\tilde \xi}_1 ) dw \right) \times \left(\int_{\mathbb{S}^4}e^{4u}  (1+ {\tilde \xi}_1 ) dw \right)\\
=&(\int_{\mathbb{S}^4}e^{4u}dw)^2- (\int_{\mathbb{S}^4}e^{4u} {\tilde \xi}_1 dw)^2
\end{align*}
and for $i=2, 3, 4, 5$
$$
\int_{\mathbb{S}^4}e^{4v} \xi_i dw
=\int_{\mathbb{S}^4}e^{4u(\phi_t(\xi))}  |det (d\phi_t)(\xi)|^{1+\frac14} \xi_i   dw
=\int_{\mathbb{S}^4}e^{4u} {\tilde \xi}_i dw.
$$
Therefore,  \eqref{invariance-expon} holds.
This completes the proof.
\end{proof}

When $P=P_0$ is chosen to coincide with the direction of the center of mass of $e^{4u}$, we also  observe  from the above proof  that
$$
\int_{\mathbb{S}^4}e^{4v} \xi_i dw=
  \int_{\mathbb{S}^4}e^{4u}  {\tilde \xi}_i dw=0, \quad i=2, 3, 4, 5
$$
  and
$$
\int_{\mathbb{S}^4}e^{4v} \xi_1 dw=
\frac{1}{\sqrt{1-a^2}}  \int_{\mathbb{S}^4}e^{4u}  (a + {\tilde \xi}_1 ) dw=0
$$
if we also choose
$
a= -\frac{\int_{\mathbb{S}^4}e^{4u} {\tilde \xi}_1  dw}{ \int_{\mathbb{S}^4}e^{4u} dw }.
$

Then, for any $u \in H^2(\mathbb{S}^4)$,  there is a $\phi_{P, t}$ such that
$$
v(\xi)= u(\phi_{P,t}(\xi)) + \frac{5}{4^2} \ln |det(d \phi_{P,t})|, \quad \xi \in \mathbb{S}^4
$$
belongs to $ \mathfrak L$.  Moreover, we have that
$\J_{\alpha} (u)= J_{\alpha}(v)$ for  $ v \in \mathfrak L$.

Then Theorem \ref{Szego} follows
immediately from Theorem \ref{main1} and Proposition \ref{invariance}.

We note that  a similar but more general Szeg\"o limit theorem for $u\in H^1(\mathbb{S}^2)$ is proven in \cite{CG} using a  variational  method with a mass center constraint, in combination with the improved Moser-Trudinger inequality in \cite{GM}.  In general,   similar Szeg\"o limit theorem should be true  for $\mathbb{S}^n, n\ge 5$ with $\alpha=\frac45$ replaced by $\alpha=\frac{n}{n+1}$, provided  that an improved Beckner's inequality could be proven for  $\alpha \le \frac{n}{n+1}$.  Note that  a counter part of   Proposition \ref{invariance} always holds for general $\mathbb{S}^n$.

\vskip4mm
{\section{Pohozaev-type Identities and Classification Result}}
 \setcounter{equation}{0}
\par

Pohozaev-type identities are  very powerful tools in studying the symmetry of solutions to semilinear elliptic equations. They play a vital role in proving classification results (see, e.g., \cite{Gidas81}, \cite{Schoen88}). Recently  Shi et. al. \cite{tian19} obtain several Pohozaev-type identities  and apply them to prove the uniqueness of axially symmetric solution of mean field equation on $ \mathbb{S}^2$ for $\alpha$. In this section, we first list several useful Pohozaev-type identities corresponding to solutions of \eqref{11}, then we prove Theorem \ref{main2} based on these identities.
\vskip4mm
We now prove Proposition \ref{pro31}. Motivated by \cite{CY87,KW74}, since \eqref{11} is invariant under adding a constant, we can normalize
 $\int_{\mathbb{S}^4} e^{4u}dw=1$. Then, \eqref{11} becomes
\begin{equation}\label{31}
  \alpha P_4 u+6\left(1- e^{4u} \right)=0, \quad \xi \in  \mathbb{S}^4.
\end{equation}
 Multiply \eqref{31} by $\xi_i,\ i=1,2,\cdots,5$ and integrate to get
  \begin{equation*}
  {  \alpha\int_{\mathbb{S}^4}(P_4 u)\xi_idw=6\int_{\mathbb{S}^4}e^{4u}\xi_idw.}
  \end{equation*}
 Note that $-\Delta \xi_i=\lambda_1 x_1=4x_1$ and $P_4 \xi_i=\lambda_1(\lambda_1+2)\xi_i$, $i=1,2,\cdots,5$. We further have
 \begin{equation*}
 {  4\alpha\int_{\mathbb{S}^4} u \xi_idw= \int_{\mathbb{S}^4}e^{4u}\xi_idw.}
 \end{equation*}
On the other hand, let
\begin{equation*}
  Q=\frac{6}{\alpha}-\left(\frac{1}{\alpha}-1\right)6e^{-4u}.
\end{equation*}
Then \eqref{31} can  be written as
\begin{equation}\label{33}
  P_4u+6= Qe^{4u}.
\end{equation}
By the Kazdan-Warner condition \eqref{kw},
one obtains
\begin{equation*}
  0= 24\left(\frac{1}{\alpha}-1\right)\int_{\mathbb{S}^4}\langle\nabla u, \nabla \xi_i\rangle dw=-24\left(\frac{1}{\alpha}-1\right)\int_{\mathbb{S}^4}  u \Delta \xi_idw=96\left(\frac{1}{\alpha}-1\right)\int_{\mathbb{S}^4}  u  \xi_idw.
\end{equation*}
Therefore,
\begin{equation}\label{35}
 \int_{\mathbb{S}^4}  u  \xi_idw=0 \quad \mbox{ and }\quad\int_{\mathbb{S}^4}  e^{4u}  \xi_idw=0\ \quad i=1,2,\cdots,5
\end{equation}
whenever  $\alpha\neq1$.
Hence, we can conclude Proposition \ref{pro31}.

\begin{remark}
The identities in \eqref{35} hold true for
  \begin{equation*}
    \alpha P_n u+(n-1)!\left(1-\frac{e^{nu}}{{\int_{\mathbb{S}^n}e^{nu}}dw}\right)=0, \quad \xi \in  \mathbb{S}^n;
  \end{equation*}
 for all $n\geq 2$  by the same method. Here
\begin{equation*}
   P_n=\begin{cases}
       \prod_{k=0}^{\frac{n-2}{2}}(-\Delta+k(n-k-1)),&\mbox{ for $n$ even};\\
       \left(-\Delta+\left(\frac{n-1}{2}\right)^2\right)^{1/2}\prod_{k=0}^{\frac{n-3}{2}}(-\Delta+k(n-k-1)),&\mbox{ for $n$ odd}.
   \end{cases}
\end{equation*}
\end{remark}
\vskip4mm
Note that Theorem \ref{main2} can be proved by the integral identity \eqref{223}. Here we provide a more essential proof using Pohozaev-type identy.

Now we focus on axially symmetric solutions to \eqref{31}. It is readily checked that $u$ satisfies
\begin{equation}\label{36}
   \alpha (1-x^2) \left[  (1-x^2)^2u'\right]'''+6 (1-x^2) (1- e^{4u})=0,\quad x\in(-1,1).
\end{equation}
Multiplying \eqref{36} by $F_2=\frac{1}{4}(5x^2-1)$ and integrating, we get
\begin{equation*}
\alpha \int_{-1}^1\left[  (1-x^2)^2u'\right]'''(1-x^2)F_2-6  \int_{-1}^1(1-x^2) e^{4u}F_2=0,
\end{equation*}
or
\begin{equation*}
  \alpha \int_{\mathbb{S}^4}(P_4u) F_2dw -6  \int_{\mathbb{S}^4} e^{4u}F_2dw =0.
\end{equation*}
Note that
\begin{equation*}
  \alpha \int_{\mathbb{S}^4}(P_4u) F_2dw=\alpha \int_{\mathbb{S}^4} u P_4 F_2dw=\alpha \lambda_2(\lambda_2+2)\int_{\mathbb{S}^4}u F_2dw=120 \alpha \int_{\mathbb{S}^4}u F_2dw.
\end{equation*}
Therefore,
\begin{equation}\label{mass}
  \int_{\mathbb{S}^4}u F_2dw=\frac{1}{20\alpha}  \int_{\mathbb{S}^4}e^{4u} F_2dw.
\end{equation}
Multiply \eqref{31} by $ \langle\nabla u,\nabla F_2\rangle$ and integrate,
\begin{equation}\label{311}
  \begin{aligned}
     \alpha\int_{\mathbb{S}^4}(P_4 u)\langle\nabla u,\nabla F_2\rangle dw &=6\int_{\mathbb{S}^4}(e^{4u}-1)\langle\nabla u,\nabla F_2\rangle dw\\
     &=6\int_{\mathbb{S}^4}\frac{1}{4}  \langle \nabla  e^{4u},\nabla F_2\rangle-6\int_{\mathbb{S}^4} \langle\nabla u,\nabla F_2\rangle dw\\
     &= 6\int_{\mathbb{S}^4}\left( -\frac{1}{4}e^{4u} \Delta F_2 +  u \Delta F_2 \right)dw\\
     &=3\int_{\mathbb{S}^4}\left(  5-\frac{1}{\alpha}   \right)e^{4u}  F_2dw.
  \end{aligned}
\end{equation}
\par
Direct computations show that in the spherical coordinate
\begin{equation*}
  \begin{aligned}
    \nabla u=((1-x^2)u', 0, 0, 0),\qquad \nabla F_2=((1-x^2)F_2', 0, 0, 0);\\
   \langle\nabla u,\nabla F_2\rangle=g_{11}(1-x^2)^2u'F_2'=(1-x^2)u'F_2',
  \end{aligned}
\end{equation*}
which together with \eqref{311} imply  that
\begin{equation}\label{p1}
 \int_{\mathbb{S}^4}(P_4 u)\langle\nabla u,\nabla F_2\rangle dw=\frac{5}{2}\int_{-1}^1\left[  (1-x^2)^2u'\right]'''x(1-x^2)^2u'.
\end{equation}
Applying integration by parts to \eqref{p1}, we get
\begin{equation}\label{312}
\begin{aligned}
  \int_{\mathbb{S}^4}(P_4 u)\langle\nabla u,\nabla F_2\rangle dw&=- \frac{15}{8}\int_{-1}^1\left[  (1-x^2)^2u'\right]''\left[(1-x^2)^2u'+x((1-x^2)^2u')'\right]\\
  &=\frac{15}{8}\int_{-1}^1|((1-x^2)^2u')'|^2 -\frac{5x}{4}\big|[(1-x^2)^2u']'\big|^2\Big|_{-1}^1+\frac{15}{16}\int_{-1}^1|((1-x^2)^2u')'|^2\\
  &=\frac{45}{16}\int_{-1}^1|((1-x^2)^2u')'|^2.
\end{aligned}
\end{equation}
Hence we obtain the following Pohozaev type inequality for solutions to \eqref{36}.
\begin{equation}\label{PI}
\int_{-1}^1|((1-x^2)^2u')'|^2= \frac{4}{5\alpha} \left(  5-\frac{1}{\alpha}   \right) { \int_{-1}^1 (1-x^2)e^{4u}  F_2. }
\end{equation}
This is equivalent to \eqref{223}.
When $\alpha=\frac{1}{5}$,  it follows that
 $$
 \int_{-1}^1|((1-x^2)^2u')'|^2=0.
 $$
 We further have
 \begin{equation*}
   ((1-x^2)^2u')' \equiv0,\quad x\in(1-,1).
 \end{equation*}
   Therefore,
   \begin{equation*}
    (1-x^2)^2u'\equiv C,\quad x\in(1-,1)
   \end{equation*}
   for some constants $C$.   Since the term $ (1-x^2)u'$ is bounded on $(-1,1)$, we have $C=0$.

  Finally
   \begin{equation*}
     u'\equiv 0 \mbox{ and so }u\equiv  \mbox{Constant}, \quad x\in(1-,1).
   \end{equation*}
   Theorem \ref{main2} has been proven.

 \vskip4mm
{\section{ Bifurcation}}
\par
In this section we shall obtain results on bifurcation curves to \eqref{15} in general for $\alpha>0$ and in particular for $\alpha  \in (\frac15, \frac12)$.
We shall first apply  the standard bifurcation theory to analyze  the local bifurcation diagram.   Let us recall the following general theorem.
\par
\begin{theorem}{\rm(\cite[Theorem 1.7]{CR71})}\label{local}
 Let $X$,$Y$ be  Hilbert  spaces, $V$ a neighborhood of $0$ in $X$ and  $F:(-1,1)\times V\rightarrow Y$ a map with the following properties:
\begin{itemize}
\item[(1)]  $F(t,0)=0$ for any $t$;
\item[(2)]  $\partial_{t}F$,$\partial_{x}F$ and $\partial_{t,x}^{2}F$ exist and are continuous;
\item[(3)] $\ker(\partial_{x}F(0,0))=\mbox{span}\{w_{0}\}$ and $Y/\mathcal{R}(\partial_{x}F(0,0))$  are one-dimensional;
\item[(4)] $\partial_{t,x}^{2}F(0,0)w_0\not\in\mathcal{R}(\partial_{x}F(0,0))$.
\end{itemize}
If $Z$ is any complement of $ \ker(\partial_{x}F(0,0))$ in $X$.
Then there exists $\varepsilon_{0}>0$, a neighborhood  of $(0,0)$ in $U\subset(-1,1)\times X$ and continuously differentiable maps $\eta:(-\varepsilon_{0},\varepsilon_{0})\rightarrow\mathbb{R}$ and $z:(-\varepsilon_{0},\varepsilon_{0})\rightarrow Z$ such that $\eta(0) =0,\
z(0)=0$ and
\begin{equation}  \nonumber
F^{-1}(0)\cap U\setminus((-1,1)\times\{0\})=\{(\eta(\varepsilon),\varepsilon w_{0}+\varepsilon z(\varepsilon))\mid\varepsilon\in(-\varepsilon_{0},\varepsilon_{0})\}.\\
\end{equation}
\end{theorem}

%
%
%
%

Recall that the  shape of the above local  bifurcating branch can be further described  by the following theorem (see, e.g., \cite[I.6]{KH12}):
\begin{theorem}\label{local2}
In the setting of  Theorem \ref{local},  let $\psi\neq0\in Y^{-1}$  satisfy
\begin{equation*}
  \mathcal{R}(\partial_{x}F(0,0))=\{y\in Y\mid\langle\psi,y\rangle=0\},
\end{equation*}
   where $Y^{-1}$ is the dual space of $Y$.  Then we have
    \begin{equation*}\label{derivative}
 \eta'(0) =-\frac{\langle\partial^{2}_{x,x}F(0,0)[w_{0},w_{0}],
 \psi\rangle}{2\|w_0\|\langle\partial^{2}_{t,x}F(0,0)w_{0},\psi\rangle}.
\end{equation*}
  Furthermore,  the  bifurcation is transcritical provided that   $ \eta'(0)\neq 0$.

\end{theorem}
\vskip 2mm
\par

Note that critical points of $I_\alpha(u)$,
satisfy
\begin{equation}\label{41}
  (1-x^2) \left[  (1-x^2)^2u'\right]'''+\rho(1-x^2)\bigl(1-\frac43 \frac{ e^{4u}}{\int_{-1}^1(1-x^2)e^{4u}}\bigr)=0, \quad x \in (-1, 1),
\end{equation}
where $\rho=\frac{6}{\alpha}$.

  Let
 \begin{equation*}
   \mathcal{V}=\bigg\{u\in H^4(\mathbb{S}^4):u=u(x),\ \int_{\mathbb{S}^4}udw=0 \bigg\}; \quad
   \mathcal{W}=\bigg\{u\in  L^2(\mathbb{S}^4):u=u(x),\ \int_{\mathbb{S}^4}udw=0\bigg\}.
 \end{equation*}
 To apply Theorem \ref{local}, we define  a nonlinear operator $\T:\R\times  \mathcal{V}\rightarrow  \mathcal{W}$ as
\begin{equation*}
 \T(\rho,u)=P_4 u+\rho \left(1- \frac{e^{4u}}{\int_{\mathbb{S}^4} e^{4u}dw}\right),
\end{equation*}

Obviously, the operator $\mathcal{T}$ is well defined.  After direct computations, one has
\begin{equation*}
    \partial_{u}\mathcal{T}(\rho,0)\phi=P_4\phi-4\rho\phi.
\end{equation*}

Define
\begin{equation*}
  \F(\rho,u)=u+\rho P_4^{-1} \left(1- \frac{e^{4u}}{\int_{\mathbb{S}^4} e^{4u}dw}\right),
\quad \mathcal{G}(u)=P_4^{-1} \left(1- \frac{e^{4u}}{\int_{\mathbb{S}^4} e^{4u}dw}\right).
\end{equation*}

 Let
 $\mathcal{S}$  denote the closure of the set of nontrivial solutions of
 \begin{equation}\label{47}
    \F(\rho,u)=0.
 \end{equation}
 It is clear that \eqref{47} and \eqref{41} are equivalent and a solution of \eqref{47}.

\begin{lemma}\label{local-condition}
 Let $ \rho_k=\frac{(k+3)!}{4(k-1)!}$ for  $k=1,2,3,\dots$, then the kernel of  $\partial_{u}\T(\rho_k,0)\phi=0$  is $1$-dimensional and
  \begin{equation}\label{43}
   \ker(\partial_{u}\T(\rho_k,0))= span\{F_k\}.
 \end{equation}
 Moreover,  the range of the operator $\partial_{u}\T(\rho_{k},0))$ is given  by
 \begin{equation}\label{44}
   \mathcal{R}(\partial_{u}\T(\rho_{k},0))=\left\{\varphi\in L^{2}(-1,1):\int_{-1}^1 (1-x^2)\varphi F_k=0\right\},
 \end{equation}
  and it has co-dimension $1$.  In addition,  we have
  \begin{equation}\label{45}
  \partial^{2}_{\rho,u}\T(\rho_{k},0)F_k\not\in\mathcal{R}(\partial_{u}\T(\rho_{k},0)).
  \end{equation}
  \end{lemma}
  \vskip 2mm
\par
  \begin{proof} We can choose
  \begin{equation*}
    X= \mathcal{V} \mbox{  and } Y=  \mathcal{W}.
  \end{equation*}
It is easy to compute  that
$$\partial_{u}\T(\rho_{k},0)\phi=P_4\phi-4\rho_k\phi, \quad
\partial^2_{uu}\T(\rho_{k},0)(\phi, \phi)= -16 \rho_k  \phi^2+16 \rho_k\int_{-1}^1 (1-x^2) \phi^2.
$$
Then   \eqref{43} follows from  \eqref{22}.  From the orthogonal property \eqref{ortho},  we deduce that
\begin{equation*}
\mathcal{R}(\partial_{u}\T(\rho_{k},0)) \mbox{ coincides with the orthogonal of } \ker(\partial_{u}\T(\rho_k,0)).
\end{equation*}
Note $ker(\partial_u \F) = ker(\partial_u \T )$.
Differentiating  $\partial_{u}\T$ with respect to $\rho$ at the point $(\rho_{k},0)$, we get
$$\partial^{2}_{\rho,u}\T(\rho_{k},0)\phi=-4\phi,$$
which,   combined with    the relation $\int_{-1}^1(1-x^2)F_k^2\neq 0$,  suggests  \eqref{45}.  The lemma is proven.
\end{proof}

For $k\in \N^+$, the following local  bifurcation result  is an immediate consequence of Theorem \ref{local} and Lemma \ref{local-condition}.
\begin{theorem}\label{bifur1}
  Let $ \rho_k=\frac{(k+3)!}{4 \cdot (k-1)!}$ for  $k=1,2,3,\dots$, the points $(\rho_{k},0)$ are  bifurcation points
 for the curve of solutions $(\rho,0)$. In particular,   there exists $\varepsilon_{0}>0$ and   continuously differentiable functions  $\rho_k:(-\varepsilon_{0},\varepsilon_{0})\rightarrow \R$ and $\psi_k:(-\varepsilon_{0},\varepsilon_{0})\rightarrow \{F_k\}^\bot$ such that $\rho_k(0)=\rho_k$, $\psi_k(0)=0$ and every
nontrivial solution of  \eqref{41} in a small neighborhood of $(\rho_{k},0)$ is of the form
\begin{equation*}
  (\rho_k(\varepsilon),\varepsilon F_k+\varepsilon \psi_k(\varepsilon)).
\end{equation*}
In particular, when $k=2$,  the bifurcation point $(\rho_2, 0)=(30, 0) $ is a transcritical bifurcation point.  Indeed, we have
\begin{equation*}\label{trans}
\rho_2'(0)= -60  {\frac{ \int_{-1}^{1} (1-x^2)F_2^3}{ \int_{-1}^{1} (1-x^2)F_2^2}}= -20.
\end{equation*}
\end{theorem}

\begin{corollary}
  Let $ \alpha_k=\frac{24\cdot (k-1)!}{(k+3)!}$ for  $k=1, 2,3,\dots$, the points $(\alpha_{k},0)$ are  bifurcation points for  the curve of solutions $(\alpha,0)$  of \eqref{15}.  Moreover, when $k=2$,  the bifurcation point $(\frac15, 0) $  is a  transcritical bifurcation point.
 \end{corollary}

\begin{remark}
When $k=1$,  the bifurcation leads to the  family of solutions $u=-\ln(1-ax),  a \in (-1, 1)$ and $\rho=6$.
 It is clear that $(\rho_k, 0)$ is not a transcritical bifurcation point for $k $ odd since $F_k$ is an odd function and  $\rho'(0)=0$ in this case.  It should be true that $(\rho_k, 0)$ is a transcritical bifurcation point for $k $ even,  we only need to check if  {$ \int_{-1}^{1} (1-x^2)F_k^3 \not =0$} in this case, which can be confirmed for small $k$ numerically.   However,  in this paper we only need to use the transcriticality of $(\rho_2, 0)$.
 \end{remark}

In order to analyze  the global  bifurcation diagram,  we employ  a  global bifurcation theorem  via degree arguments (see \cite{KH12,R}) and also exploit  special properties of solutions to \eqref{41}.

First, we recall   a global  bifurcation result (see \cite[Theorem II.5.8]{KH12}).

 \begin{proposition}\label{t47}
  In Theorem \ref{bifur1},   the bifurcation at $(\rho_{k},0)$ is global
and  satisfies the Rabinowitz alternative, i.e., a global continuum of solutions to \eqref{41} either goes
to infinity  in $R \times \mathcal{W}$ or meets the trivial solution curve  at $(\rho_m, 0)$  for some $m \ge 1$ and $m\neq k$.
 \end{proposition}

 Next we state and prove  the following  more specific  global bifurcation result regarding \eqref{41}.

 \begin{theorem}\label{main-bifur}
 1)   For $k\ge 2$,  there exists a global continuum of solutions
 $\mathcal{B}^+_k  \subset \mathcal{S} \setminus \{ (\rho, 0), \rho \in \mathbb{R}\}$
  of \eqref{41}  which coincides  in a small neighborhood of $(\rho_k, 0)$ with
  $$\{ (\rho_k(\varepsilon),\varepsilon F_k+\varepsilon \psi_k(\varepsilon)),   \varepsilon<0\}.$$
    $\mathcal{B}^+_k $ is contained in
  $\mathcal{N}_2:= \{ (\rho, u):  \rho >  30, \,  u \in L^2(-1, 1) \}$ and  is uniformly bounded in $L^2(-1, 1)$  for $\rho$ in
  any fixed  finite interval $[\rho_m, \rho_M] \subset (30, \infty)$.  Furthermore, $ \mathcal{B}^+_k $  satisfies  the  improved Rabinowitz alternative, i.e.,  either $\mathcal{B}^+_k$ extends in $\rho$ to infinity  or meets the trivial solution curve at $(\rho_m, 0)$  for some $m \geq 2$.

 2)  Similarly, for $k\ge 2$,  there exists   a global continuum of solutions  $\mathcal{B}^{-}_k $  which coincides  in a small neighborhood of $(\rho_k, 0)$ with  $\{ (\rho_k(\varepsilon),\varepsilon F_k+\varepsilon \psi_k(\varepsilon)),   \varepsilon>0\} $.  When $k\ge 3$,  $\mathcal{B}^{-}_k $ is contained in
  $\mathcal{N}_2$  and satisfies  the  boundedness for  $\rho$ in any fixed  finite interval  $[\rho_m, \rho_M] \subset (12, \infty)$.  Furthermore,  the improved Rabinowitz alternative holds.

 3)  Moreover,    $\mathcal{B}^+_k=\{ u:  u(x)=v(-x),  v \in \mathcal{B}^{-}_k\} $  when $k$ is odd.

 4)  The global continuum of solutions $\mathcal{B}^{-}_2$ of \eqref{41} must be contained in the set
 $$\mathcal{N}_1:= \left\{ (\rho, u):  \rho \in \left( \frac{6\times 1800}{ 473 + \sqrt{209329}},  30\right) \supset (12, 30), \,  u \in L^2(-1, 1)\right \}.$$
    Furthermore,   $\mathcal{B}^{-}_2$ is unbounded in  $L^\infty ([-1, 1] )$,  and  there exists a sequence of $ (\rho^{(k)}, u^{(k)}) \in \mathcal{B}^{-}_2,  k =1, 2, \cdots $ such that $ \rho^{(k)} \to \frac12$ and
 $\|u^{(k)}\|_{L^\infty([-1,1])} \to \infty$.  As an immediate consequence, there is a nontrivial solution to \eqref{41} for any $\rho \in (12, 30)$.
  \end{theorem}

 \begin{proof}
  To prove 1) and 2),  we only need to first apply the general global bifurcation theory  and then  use Theorem \ref{main2} to  show
  $\mathcal{B}^+_k $ and $\mathcal{B}^{-}_k $  are contained in $\mathcal{N}_2:= \{ (\rho, u):  \rho >  30, \,  u \in L^2(-1, 1) \}$.

  A  general compactness result (\cite[Theorem 1.1]{M06})   says that the solutions to \eqref{41}   can only blow up in $L^\infty([-1, 1])$ at  $\rho=6k$ for an  positive integer $k$  when \eqref{41} is considered  as an  fourth order  Q-curvature type equation on $\mathbb{S}^4$,  and $k$ is the number of blowup points. (See also  \cite[Theorem 4.3]{Hang}    from  a  view point of constrained inequalities.)  Since an axially symmetric
  solution can blow up at most two points  at a finite parameter $\rho$,   we must  have $k=1, 2$.   Therefore, this leads to the boundedness of
  $\mathcal{B}^+_k ,  k\ge 2 $  and $\mathcal{B}^{-}_k,  k\ge 3$ for  $\rho$ in  any fixed  finite interval $[\rho_m, \rho_M] \subset (30, \infty)$.

 To prove 3),  we note that $u(x)=v(-x)$ is a solution to \eqref{41} if so is $v(x)$,  and $u(x) $ is not an even function for $u \in \{ u:  u(x)=v(-x),  v \in \mathcal{B}^{-}_k\} $  near the bifurcation point $\{(\rho_k, 0)\}$.   Therefore, by the local bifurcation result Theorem \ref{bifur1},  we know  $ \{ u:  u(x)=v(-x),  v \in \mathcal{B}^{-}_k\} $  is different  from  $ \mathcal{B}^{-}_k  $  and hence  coincides with $\mathcal{B}^+_k$  near $\{(\rho_k, 0)\}$.  Then  3) follows immediately.

 To prove 4),  we  first use the transcriticality  \eqref{trans} to get  $\mathcal{B}^{-}_2 \cap \mathcal{N} _1 \not = \emptyset$.
 By Theorem \ref{main1}-\ref{main2},  we conclude $\mathcal{B}^{-}_2 \subset \mathcal{N}_1$.   Since there is no other bifurcation points
for $\rho$   between $ \frac{6\times 1800}{ 473 + \sqrt{209329}} >6$ and $30$ ,  and $\rho=12$ is the only blowup point,   we conclude that
$ \mathcal{B}^{-}_2 $ must go to infinity in $\mathcal{W}$ and in $L^\infty([-1,1])$ at $\rho=12$.

 This completes the proof.
 \end{proof}

\begin{remark}
The above theorem implies that $ \mathcal{B}^{-}_2 $  does not coincide with other bifurcation branches.
It would be interesting to see whether the solution branches bifurcating from different points $(\rho_k, 0)$ coincide with each other or not, i.e., whether  $\mathcal{B}^+_k =  \mathcal{B}^+_m $  or  $\mathcal{B}^+_k = \mathcal{B}^{-}_m $ in general for any $m \not = k$, or particularly  for $m\equiv k+1$ $(mod \ 2)$.   Also  it is not clear  whether $\mathcal{B}^+_k =  \mathcal{B}^{-}_k $ for some $k \ge 3$.
\end{remark}

{\bf Proof of Theorem \ref{main3}}:  Theorem \ref{main3} follows immediately from Theorem \ref{main-bifur}.  This leads to the existence of a nontrivial solution to \eqref{15} for $\alpha \in (\frac15, \frac12)$.  \qed

\vskip4mm
{\section{Discussion}}
 \setcounter{equation}{0}
\par
In this Section, we shall discuss some ideas to close the gap $\alpha\in(\frac12, \frac{473 + \sqrt{209329}}{1800})$.

Note that Gui and Wei \cite{Gui-Wei} used an induction method to show
\begin{equation*}
 \frac{1}{\alpha}-\beta\leq \frac{4}{\lambda_n} \mbox{ for all } n\geq 4
\end{equation*}
with the sequence $\lambda_n \to \infty.$
So it follows $\frac{1}{\alpha}-\beta\rightarrow0$ as $n\rightarrow\infty$, which leads to a contradiction. Following  the arguments in  \cite{Gui-Wei}, we    divide \eqref{242} by $\lambda_{n+1}$ to get
  \begin{equation}\label{245}
\begin{aligned}
 &\quad \frac{5\beta}{3\lambda_{n+1}}\left[-\frac{8}{\alpha^2}+ (15\lambda_{n+1}+136)\frac{1}{\alpha}-180-30\lambda_{n+1}\right]\\
&\geq\left(\frac{1}{\alpha}-\beta\right)\left[5\beta\left(3-\frac{12}{\lambda_{n+1}}+\frac{32}{5\alpha\lambda_{n+1}} \right)-\frac{3\bar c_n}{\lambda_{n+1}}\left(\frac{1}{\alpha}-\beta\right)\right].
\end{aligned}
\end{equation}
 A direct calculation shows that
\begin{equation}\label{246}
  \mbox{LHS of } \eqref{245}  \leq \frac{100\beta}{\lambda_{n+1}},
\end{equation}
which is   the basic ingredient  for the induction procedure in \cite{Gui-Wei}.  Next, the major task is to find an  appropriate $d$ so that
\begin{equation*}
 \frac{1}{\alpha}-\beta\leq \frac{d}{\lambda_n} \mbox{ for all } n\geq n_0.
\end{equation*}
However, we do not know what  the initial value $n_0$ should be, which is dependent on the choice of $d$.   We assume, by   induction, that
   \begin{equation}\label{248}
   \frac{1}{\alpha}-\beta\leq \frac{d}{\lambda_n}
   \end{equation}
 for some $n\geq n_0$. Then
\begin{equation}\label{249}
 \frac{1}{\alpha}-\beta\leq \frac{d}{\lambda_{n+1}}
\end{equation}
must hold.
It follows from \eqref{248} that
\begin{equation*}
 \mbox{RHS of }\eqref{245} \geq\left(\frac{1}{\alpha}-\beta\right)\left[5\beta\left(3-\frac{12}{\lambda_{n+1}}+\frac{32}{5\alpha\lambda_{n+1}} \right)-\frac{3\bar c_n}{\lambda_{n+1}} \frac{d}{\lambda_n} \right].
\end{equation*}
To ensure \eqref{249}, it suffices from  \eqref{246} to prove that
\begin{equation}\label{250}
5\beta\left(3-\frac{12}{\lambda_{n+1}}+\frac{32}{5\alpha\lambda_{n+1}} \right)-\frac{3\bar c_n}{\lambda_{n+1}} \frac{d}{\lambda_n} \geq \frac{100}{d}\beta
.
\end{equation}
For $n$ large, this requires
\begin{equation}\label{252}
   \frac{15d-100}{\alpha}\geq \frac{3d^2}{2}
\end{equation}
for $\alpha\in(\frac 12,\frac{473 + \sqrt{209329}}{1800})$. However,  there does not exist such a constant $d$ for   \eqref{252} to hold.   Hence, this method does not seem to yield the optimal constant $\alpha=\frac12$ for this problem.
\begin{remark}
 We also intend to replace  denominator  in  \eqref{248} and \eqref{249} by  $\lambda_n^t$ and $\lambda_{n+1}^t$, respectively, for some $ t>0$.
For this purpose we  only  need to  slightly modify  the  previous procedure.  After some calculations,  \eqref{250} becomes
 \begin{equation*}
5\beta\left(3-\frac{12}{\lambda_{n+1}}+\frac{32}{5\alpha\lambda_{n+1}} \right)-\frac{3\bar c_n}{\lambda_{n+1}} \frac{d}{\lambda_n^t} \geq \frac{100\lambda_n^{t-1}}{d}\beta
 \end{equation*}
Therefore, we need to show that  for  $\alpha\in(\frac 12,\frac{473 + \sqrt{209329}}{1800})$ and $n$ large,
\begin{equation*}
  \frac{1}{\alpha}\left[15d\lambda_n^{t+1}-100\lambda_n^{2t}\right]\geq \frac{3d^2}{2}\lambda_{n+1}\lambda_n +o_n(1),
\end{equation*}
 which suggests that $t=1$ is the best choice, since $\lambda_n\rightarrow\infty$ as $n\rightarrow\infty$.
\end{remark}
\par
\vskip3mm
 It is worth pointing out  that the inequality \eqref{key} ensured by Lemma \ref{132}  plays  an important role in the proof of Theorem \ref{main1}.  In view of \eqref{214},   we may want to estimate $M$ more accurately by considering
\begin{equation}\label{253}
 M=\max_{x\in[-1,1]}G'(x)=\frac{1}{\alpha}\left(1-\frac{4}{3\gamma}e^{4u}\right).
\end{equation}
To this purpose,  first, we observe, due to $G'(x)\leq M\leq\frac{1}{\alpha}$,  that
 \begin{equation}\label{254}
  -\frac{1}{\alpha}(1-x)\leq G(x)\leq \frac{1}{\alpha}(1+x),\quad x\in[-1,1].
 \end{equation}
 This leads to
 \begin{equation}\label{255}
    -\frac{1}{\alpha(1+x)}\leq u'(x)\leq \frac{1}{\alpha(1-x)},\quad x\in(-1,1).
 \end{equation}
 Assume, without loss of generality, $u(0)=0$. Then one has
 \begin{equation}\label{256}
   u(x)\geq -\frac{1}{\alpha}\ln2,\quad e^{4u}\geq2^{-\frac{4}{\alpha}}.
 \end{equation}
We next estimate $\gamma$ from above. Note that   $u=-\ln(1-ax) $ is a solution when $\alpha=1$. By some computations, we can see that \eqref{254}--\eqref{255} can not be improved, and there is {\it no}  apriori estimate for $\gamma$. However, if we assume
 \begin{equation}\label{257}
   1-\alpha\beta>0,
 \end{equation}
 then we can use \eqref{221} to estimate  $\gamma$.   Precisely,  take  a constant  $a$  such  that
\begin{equation}\label{258}
  1-a^2=c(1-\alpha\beta),\ c\in\left(0,\frac{4}{5}\right).
\end{equation}
 Then
\begin{equation}\label{259}
 \frac{4}{5}(1-\alpha\beta)\leq \frac{2}{\gamma}\int_0^a (1-x^2)(a^2-x^2)(1-x)^{-4/\alpha}+(1-a^2).
\end{equation}
  Since $\alpha\in(0.5,0.517)$,  one has
 \begin{equation*}
 \begin{aligned}
   I&:=\int_0^a (1-x^2)(a^2-x^2)(1-x)^{-4/\alpha}\\
   &\leq\int_0^a (1+x)(a^2-x^2)(1-x)^{-7}\\
  &=\frac{2a^2}{15}\left((1-a)^{-5}-1\right).
 \end{aligned}
\end{equation*}
It follows from \eqref{258} that  $a=\sqrt{1-c(1-\alpha\beta)}$ and so
\begin{equation*}
  \frac{1}{1-a}=\frac{1+a}{c(1-\alpha\beta)}\leq \frac{2}{c(1-\alpha\beta)}.
\end{equation*}
We further compute directly
\begin{equation*}
 \begin{aligned}
   I&\leq \frac{2}{15}(1-c(1-\alpha\beta))\left[\left(\frac{2}{c(1-\alpha\beta)}\right)^5-1\right]\\
   &\leq \frac{2^6[1-c(1-\alpha\beta)]}{15[c(1-\alpha\beta)]^5},
 \end{aligned}
\end{equation*}
which,  joint with \eqref{259},   lead to
\begin{equation}\label{260}
  \gamma\leq \left[(\frac{4}{5}-c)(1-\alpha\beta)\right]^{-1}\frac{2^7[1-c(1-\alpha\beta)]}{15[c(1-\alpha\beta)]^5}.
\end{equation}
We see from \eqref{253}, \eqref{256} and \eqref{260} that
\begin{equation}\label{262}
\begin{aligned}
   M&\leq \frac{1}{\alpha}\left[1- \frac{5c^5}{2^{13}}\left(\frac{4}{5}-c\right)\frac{(1-\alpha\beta)^6}{1-c(1-\alpha\beta)}\right]\\
&:=\frac{1}{\alpha}\left[1- B(\alpha,\beta)\right].
\end{aligned}
\end{equation}
We conclude from the above relations that
the upper bound of $M$ can be slightly improved in terms of $B(\alpha, \beta)>0$ by choosing, e.g.  $c=\frac25$, given $\beta \not =0$ and $1-\alpha \beta >0$. Thus, one has
\begin{equation}\label{B}
  B(\alpha,\beta)=\frac{1}{5^42^7} \frac{(1-\alpha\beta)^6}{3+2\alpha\beta}.
\end{equation}

We next will use some notations in Section 3 and  assume that
\begin{equation}\label{range1}
 0.5165\leq\alpha<0.51696.
\end{equation}
Note that
\begin{equation}\label{beta-1}
 \beta>h(0.5165)> 1.3375.
\end{equation}
It follows from \eqref{upbd} and \eqref{257} that
\begin{equation}\label{bd-1}
 1 >\alpha\beta>0.5165h(0.5165)>0.69057
\end{equation}
and
\begin{equation}\label{upbd-1}
   \frac{1}{\alpha}-\beta\leq 0.255  \beta.
\end{equation}
\par
 Let   $t=\frac{1}{\alpha}-\beta$. We now estimate the lower bound of $t$.
Noting that
\begin{equation*}
  f_n(\alpha)<f_6(0.5165)<-13.764,\quad \mbox{for }n\geq6,
\end{equation*}
we derive from \eqref{242} and \eqref{beta-1}  that
 \begin{equation*}
 \begin{aligned}
   \frac{5}{3}f_6(0.5165)&>\frac{5}{3}f_6(\alpha)\geq
   A_6t-\frac{B_6}{\beta} t^2> A_6t-\frac{B_6}{1.3375} t^2.
 \end{aligned}
 \end{equation*}
 Noting that $t>0$, a direct calculation indicates that
 \begin{equation}\label{t-lowbd1}
   t>0.253,
 \end{equation}
which  joint with \eqref{bd-1} and \eqref{upbd-1} suggests that
\begin{equation}\label{B-bd1}
\begin{aligned}
   7.85*10^{-10}&>\frac{0.255^6}{5^42^7(3+2*0.69057)}\\
   &\ge   B(\alpha,\beta)= \frac{1}{5^42^7} \frac{(1-\alpha\beta)^6}{3+2\alpha\beta}\\
 &>\frac{(0.5165*0.253)^6}{5^5 2^7}  >1.132*10^{-11}.
\end{aligned}
\end{equation}

\par
On the other hand,  we need to modify some inequalities  in Section $3$ by exploiting \eqref{262} instead of \eqref{21}. First,  the inequality  \eqref{218} becomes
\begin{equation*}
      \begin{aligned}
    \left(\frac{24}{\alpha}-12\right)&\int_{-1}^1 (1-x^2)G^2\geq \lfloor G\rfloor^2+ \left(6-\frac{4}{\alpha}(1-B)\right)\int_{-1}^1(1-x^2)^2(G')^2.
\end{aligned}
\end{equation*}
 Here,  $B$ denotes $B(\alpha,\beta)$. Similarly, we have
\begin{equation*}
     \lfloor G\rfloor^2\leq \left(\frac{4}{\alpha}(1-B)-6\right)\int_{-1}^1 |[(1-x^2)G]'|^2+\frac{16+8B}{\alpha}\int_{-1}^1 (1-x^2)G^2
\end{equation*}
and
\begin{equation}\label{264}
   \begin{aligned}
  \bar D&:=\sum_{k=3}^\infty\left[\lambda_k(\lambda_k+2)-\left(10-\frac{4+2B}{3\alpha}\right)(\lambda_k+2)-\frac{16+8B}{\alpha}
    \right]b_k^2\\
   &\leq\left(\frac{16-10B}{3\alpha}-16\right)\int_{-1}^1 |[(1-x^2)G]'|^2-\left[6\left(\frac{4+2B}{3\alpha}-6\right)-\frac{16+8B}{\alpha}\right]\frac{4\beta^2}{15} \\
    &\leq \frac{16\beta}{15}\left[
    \left(\frac{16-10B}{3\alpha}-16\right)\left(5-\frac{1}{\alpha}\right)+\left(9+\frac{2+B}{\alpha}\right)\beta\right].
\end{aligned}
\end{equation}
  From \eqref{223}, \eqref{230} and \eqref{264}, we derive  that
\begin{equation*}
  \begin{aligned}
&\quad \sum_{k=3}^n\left(\lambda_k-\lambda_{n+1}-\frac{4+2B}{15\alpha}\right) (\lambda_k+2)b_k^2+\left(\lambda_{n+1}- 10+\frac{4+2B}{5\alpha}\right)\left[\frac{16\beta}{15}\left(5-\frac{1}{\alpha}\right)-\frac{8}{5}\beta^2
    -\frac{8}{7}a_2^2\right]\\
    &=\sum_{k=3}^\infty\left[\lambda_k(\lambda_k+2)-\left(10-\frac{4+2B}{3\alpha}\right)(\lambda_k+2)-\frac{16+8B}{\alpha}
    \right]b_k^2\\
    &\leq\frac{16\beta}{15} \left[\left(9+\frac{2+B}{\alpha}\right)\beta
    +\left(\frac{16-10B}{3\alpha}-16\right)\left(5-\frac{1}{\alpha}\right)\right].
  \end{aligned}
\end{equation*}
One further obtains that
 \begin{equation}\label{265}
   \begin{aligned}
0&\leq\frac{16\beta}{15}\left(5-\frac{1}{\alpha}\right)\left(\frac{68-56B}{15\alpha}-6-\lambda_{n+1} \right)+\frac{8}{15}\beta^2
\left(3\lambda_{n+1}-12+\frac{32+16B}{5\alpha} \right) \\
&+\frac{8}{25}\left(\frac{1}{\alpha}-\beta\right)^2\left[7\left(\lambda_{n+1}- 10+\frac{4+2B}{5\alpha}\right)+ c_{n,B}\right],
  \end{aligned}
 \end{equation}
 where
\begin{equation*}
 c_{n,B}=\frac{1}{2}\lambda_{n+1}^2+\left(\frac{4+2B}{15\alpha}-14\right)\lambda_{n+1}+90-(n+16)\frac{4+2B}{15\alpha}.
\end{equation*}
 Note that \eqref{265}  is equivalent to
 \begin{equation*}
   \begin{aligned}
   0&\leq10\beta\left(5-\frac{1}{\alpha}\right)\left(\frac{68-56B}{15\alpha}-6-\lambda_{n+1} \right)+5\beta^2
\left(3\lambda_{n+1}-12+\frac{32+16B}{5\alpha} \right)+3\bar  c_{n,B}\left(\frac{1}{\alpha}-\beta\right)^2
\end{aligned}
 \end{equation*}
and
\begin{equation*}
 \bar c_{n,B}=\frac{1}{2}\lambda_{n+1}^2-7\lambda_{n+1}+20+\frac{4+2B}{15\alpha}\left(\lambda_{n+1}+5-n\right).
\end{equation*}
Therefore,
 \begin{equation}\label{B-ine}
   \begin{aligned}
    &\quad \frac{5\beta}{3}\left[(32B-8)\frac{1}{\alpha^2}+ (15\lambda_{n+1}+136-112B)\frac{1}{\alpha}-180-30\lambda_{n+1}\right]\\
&\geq\left(\frac{1}{\alpha}-\beta\right)\left[5\beta\left(3\lambda_{n+1}-12+\frac{32+16B}{5\alpha} \right)-3\bar c_{n,B}\left(\frac{1}{\alpha}-\beta\right)\right].
\end{aligned}
 \end{equation}
When $n=3$,  it follows from \eqref{upbd-1} that
 \begin{equation*}
  \begin{aligned}
 \mbox{RHS  of}  \, \eqref{B-ine}
   &=\left(\frac{1}{\alpha}-\beta\right)\beta\left[360+\frac{16(2+B)}{\alpha}-0.255\left(648+\frac{12(2+B)}{\alpha}\right)
   \right]\\
   &\geq\left(\frac{1}{\alpha}-\beta\right)\beta\left(194.76+\frac{12.94(2+B)}{\alpha}\right).
\end{aligned}
 \end{equation*}
  Combining \eqref{B-bd1},  \eqref{t-lowbd1}  and \eqref{B-ine},  we conclude
  \begin{equation*}
     \begin{aligned}
 &\quad 0.253\beta\left(194.76+\frac{12.94(2+1.132*10^{-11})}{\alpha}\right) <t\beta\left(194.76+\frac{12.94(2+B)}{\alpha}\right)\\
  &\leq \frac{5\beta}{3}\left[\frac{32B-8}{\alpha^2}+ \frac{556-112B}{\alpha}-1020\right]\\
  &\le \frac{5\beta}{3}\left[\frac{32*7.85*10^{-10}-8}{\alpha^2}+ \frac{556-112*1.132*10^{-11}}{\alpha}-1020\right].
\end{aligned}
  \end{equation*}
 Then one has
  \begin{equation*}
     \alpha<0.511.
  \end{equation*}
This is a contraction with \eqref{range1}. Consequently,
 \begin{equation*}
   \alpha< 0.5165.
 \end{equation*}
 We now assume that
\begin{equation}\label{range2}
 0.516\leq \alpha<0.5165.
\end{equation}
 Similarly, one has
 \begin{equation*}
   \frac{1}{\alpha}-\beta<0.2611\beta,\quad \beta>1.3341 \quad \mbox{ and }\quad t>0.2492.
 \end{equation*}
 Furthermore,
 \begin{equation*}
\begin{aligned}
   9.05*10^{-10}>  B(\alpha,\beta)>1.13*10^{-11}.
\end{aligned}
\end{equation*}
Then we find
 \begin{equation*}
 \frac{32*9.05*10^{-10}-8}{\alpha^2}+ \frac{556-112*1.13*10^{-11}}{\alpha}-1020-0.6*0.2492\left[190.8+ \frac{12.8(2+1.13*10^{-11})}{\alpha} \right]>0.
\end{equation*}
A direct calculation shows that
\begin{equation*}
  \alpha<0.512.
\end{equation*}
This  yields a contradiction.
\par
Repeating previous arguments, we can obtain a contradiction for the following ranges of $\alpha$:
\begin{equation*}
  \alpha\in[0.5155,0.516),\quad\alpha\in[0.515,0.5155), \quad \alpha\in[0.5145,0.515).
\end{equation*}
Hence, $\alpha<0.5145.$
\par
Unfortunately,    it does  not seem possible to improve the estimate of $\alpha$ significantly in this way and recursively, due to \eqref{B-bd1}, let alone to obtain the possible optimal constant $\alpha=\frac12$.

\vskip4mm
$\mathbf{Acknowledgement}$
This research is partially supported by  NSF grant DMS-1601885 and DMS-1901914. The first author would like to thank Xiaodong Wang for helpful conversations regarding the  Szeg\"o limit theorem and the functional invariance under the conformal transformation.

\end{document}